\documentclass[11pt,reqno,tbtags]{amsart}

\usepackage[normalem]{ulem}
\usepackage{color}

\newcommand\kolla{\relax}
\usepackage{doi}

\usepackage{amssymb}
\usepackage{natbib}
\bibpunct[, ]{[}{]}{;}{n}{,}{,}

\numberwithin{equation}{section}
\allowdisplaybreaks

\newtheorem{theorem}{Theorem}[section]
\newtheorem{lemma}[theorem]{Lemma}
\newtheorem{proposition}[theorem]{Proposition}

\theoremstyle{definition}

\newtheorem{problem}[theorem]{Problem}
\newtheorem{remark}[theorem]{Remark}

\newtheorem*{ack}{Acknowledgement}

\theoremstyle{remark}

\newenvironment{romenumerate}{\begin{enumerate}
 }{\end{enumerate}}

\newcounter{oldenumi}
{\setcounter{oldenumi}{\value{enumi}}
\begin{romenumerate} \setcounter{enumi}{\value{oldenumi}}}
{\end{romenumerate}}

\newcounter{thmenumerate}

\newcounter{xenumerate}   

\newcommand\step[2]{\vskip 10pt \noindent\emph{Step #1: #2} \noindent}

\newcommand{\refT}[1]{Theorem~\ref{#1}}

\newcommand{\refL}[1]{Lemma~\ref{#1}}
\newcommand{\refR}[1]{Remark~\ref{#1}}
\newcommand{\refS}[1]{Section~\ref{#1}}

\newcommand{\refand}[2]{\ref{#1} and~\ref{#2}}

\begingroup
  \divide\count255 by 60
  \multiply\count255 by -60
  \advance\count255 by \time
  \ifnum \count255 < 10 \xdef\klockan{\the\count1.0\the\count255}
  \else\xdef\klockan{\the\count1.\the\count255}\fi
\endgroup

\newcommand\nopf{\qed}   

\newcommand\set[1]{\ensuremath{\{#1\}}}

\newcommand\bigpar[1]{\bigl(#1\bigr)}
\newcommand\Bigpar[1]{\Bigl(#1\Bigr)}

\newcommand\lrpar[1]{\left(#1\right)}

\newcommand\Bigabs[1]{\Bigl|#1\Bigr|}

\newcommand\lrabs[1]{\left|#1\right|}
\def\rompar(#1){\textup(#1\textup)}    
\newcommand\xfrac[2]{#1/#2}

\newcommand\parfrac[2]{\lrpar{\frac{#1}{#2}}}

\def\xexp(#1){e^{#1}}
\newcommand\ceil[1]{\lceil#1\rceil}
\newcommand\floor[1]{\lfloor#1\rfloor}

\newcommand\ntoo{\ensuremath{{n\to\infty}}}

\newcommand\iid{i.i.d.\spacefactor=1000}
\newcommand\ie{i.e.\spacefactor=1000}

\newcommand\cf{cf.\spacefactor=1000}

\newcommand\whp{{w.h.p.\spacefactor=1000}}
\newcommand\whpx{w.h.p} 
\newcommand\wvhp{\whp}

\newcommand{\tend}{\longrightarrow}

\newcommand\pto{\overset{\mathrm{p}}{\tend}}

\newcommand\bbR{\mathbb R}

\newcounter{CC}
\newcommand{\CC}{\stepcounter{CC}\CCx} 
\newcommand{\CCx}{C_{\arabic{CC}}}     
\newcommand{\CCdef}[1]{\xdef#1{\CCx}}     
\newcounter{cc}
\newcommand{\cc}{\stepcounter{cc}\ccx} 
\newcommand{\ccx}{c_{\arabic{cc}}}     
\newcommand{\ccdef}[1]{\xdef#1{\ccx}}     

\newcommand\E{\operatorname{\mathbb E{}}}
\renewcommand\P{\operatorname{\mathbb P{}}}

\newcommand\Var{\operatorname{Var}}
\newcommand\Cov{\operatorname{Cov}}

\newcommand\Bi{\operatorname{Bi}}
\newcommand\Bin{\operatorname{Bin}}
\newcommand\Be{\operatorname{Be}}

\newcommand\ga{\alpha}
\newcommand\gb{\beta}
\newcommand\gd{\delta}

\newcommand\gam{\gamma}

\newcommand\gl{\lambda}

\newcommand\eps{\varepsilon}

\newcommand\cA{\mathcal A}

\newcommand\cS{{\mathcal S}}

\newcommand\ett[1]{\boldsymbol1[#1]}

\def\[#1]{[\![#1]\!]}

\newcommand\qqq{^{1/3}}

\newcommand\qw{^{-1}}
\newcommand\qww{^{-2}}

\renewcommand{\=}{:=}

\newcommand\gnd{\ensuremath{G(n,d)}}

\newcommand\gnx[1]{\ensuremath{G(n,#1)}}

\newcommand\gnxx[1]{\ensuremath{G_{n,#1}}}
\newcommand\hnxx[1]{\ensuremath{H_{n,#1}}}
\newcommand\cnxx[1]{\ensuremath{C_{n,#1}}}

\newcommand\rnxx[1]{\ensuremath{R_{n,#1}}}
\newcommand\gnp{\gnxx p}
\newcommand\hnp{\hnxx p}
\newcommand\cnp{\cnxx p}

\newcommand\Gnc{\gnxx {c/n}}
\newcommand\Hnc{\hnxx {c/n}}
\newcommand\Rnc{\rnxx {c/n}}

\newcommand\nn{[n]}

\newcommand\spr{\operatorname{spread}}

\newcommand\ellx{\nu}
\newcommand\nux{\bar\mu}
\newcommand\tZ{\tilde{Z}}
\newcommand\Cn{\mathsf{C}_n}
\newcommand\Kn{\mathsf{K}_n}
\newcommand\Pn{\mathsf{P}_n}

\newcommand\Kx{\mathsf{K}}
\newcommand\Qx{\mathsf{Q}}

\newcommand\Pc[1]{\operatorname{{\mathbb P}_c}\left\{#1\right\}}
\newcommand\Ec{\operatorname{{\mathbb E}_c}}
\newcommand\Varc{\operatorname{{Var}_c}}
\makeatletter
\newcommand\colorwave[1][blue]{\bgroup \markoverwith{\lower3.5\p@\hbox{\sixly \textcolor{#1}{\char58}}}\ULon}
\font\sixly=lasy6 
\makeatother

\begin{document}
\title{On the spread of random graphs}

\date{9 August 2012} 

\author{Louigi Addario-Berry}
\address{Department of Mathematics and Statistics, McGill University, 805 Sherbrooke Street West, 
		Montr\'eal, Qu\'ebec, H3A 2K6, Canada}
\email{louigi@gmail.com}
\urladdr{http://www.math.mcgill.ca/~louigi/}

\author{Svante Janson}
\address{Department of Mathematics, Uppsala University, PO Box 480,
SE-751~06 Uppsala, Sweden}
\email{svante.janson@math.uu.se}
\urladdr{http://www.math.uu.se/~svante/}

\author{Colin McDiarmid}
\address{Department of Statistics, University of Oxford,
1 South Parks Road, Oxford OX1 3TG, UK}
\email{cmcd@stats.ox.ac.uk}
\urladdr{http://www.stats.ox.ac.uk/~cmcd/}


\begin{abstract}
  The \emph{spread} of a connected graph $G$ was introduced by Alon,
  Boppana and Spencer~\cite{alon1998aii} and measures how tightly connected the graph is.
  It is defined as the maximum over all Lipschitz functions $f$ on $V(G)$ of the variance
  of $f(X)$ when $X$ is uniformly distributed on $V(G)$.
  We investigate the spread for certain models of sparse random graph;
  in particular for random regular graphs $\gnd$, 
  for Erd\H{os}-R\'enyi random graphs $\gnp$ in the supercritical range $p>1/n$,
  and for a `small world' model.
  For supercritical $\gnp$, we show that if $p=c/n$ with $c>1$ fixed then
  with high probability the spread of the giant component is bounded, and we prove corresponding statements for
  other models of random graphs, including a model with random edge-lengths. 
  We also give lower bounds on the spread for the barely supercritical case when $p=(1+o(1))/n$.
  Further, we show that for $d$ large, with high probability the spread of $\gnd$ becomes arbitrarily close
  to that of the complete graph $\Kn$. 
\end{abstract}
  
\maketitle


\section{Introduction}\label{S:intro}

  If $G$ is a graph, a \emph{Lipschitz function} $f$ on $G$ is a real-valued
  function defined on the vertex set $V(G)$ such that
  $|f(v)-f(w)|\le 1$ for every pair of adjacent vertices $v$ and $w$.
  We may regard a function $f:V(G)\to\bbR$ on a graph $G$ as a random
  variable by evaluating $f$ at a random, uniformly distributed, vertex.
  We may thus talk about the mean, median and variance of $f$.
  For example, if $G$ has $n$ vertices, the mean $\E f$ is $\sum_v f(v)/n$, and the variance of $f$ is
\begin{equation} \label{eqn.var}
  \frac1{n} \sum_{v} (f(v) - \E f)^2 = \frac1{2n^2} \sum_{v,w} (f(v)-f(w))^2.
\end{equation}

  For a fixed connected graph $G$, we define the \emph{spread} of $G$ 
  to be the supremum of the variance of $f$ 
  over all Lipschitz functions $f$ on $G$, 
  and we denote this quantity by  $\spr(G)$.
  (Note that the supremum would be infinite if we considered a disconnected graph.)
  The spread of a graph was introduced by Alon, Boppana and Spencer in~\cite{alon1998aii},
  and considered further in~ \cite{bht2000}.
  In particular it is shown in~\cite{alon1998aii} that the spread yields the
  optimal coefficient in the exponent in a natural asymptotic isoperimetric inequality:
  we discuss this briefly below.
  The spread is a natural measure of the overall connectivity of a graph, and the
  purpose of this paper is to investigate the spread for certain models of random graph.

\subsection{The spread of a graph}

  Before we introduce our results concerning random graphs,
  let us give some background on the spread of a graph.
  Observe first that spread is an edge-monotone function in the sense that
  if we add an edge to a graph then the set of Lipschitz functions becomes smaller,
  and thus the spread becomes smaller or remains the same.
  
  For every connected graph $G$, $\spr(G)$ is attained, so we can replace
  supremum by maximum. In fact, it is shown in Theorem~2.1 of~\cite{alon1998aii}
  that there is always a Lipschitz function $f$ achieving $\spr(G)$ which is
  integer-valued and of the  \kolla
  following simple form: if $S$ denotes the set of vertices $v$ with $f(v)=0$,
  then each component $H$ of $G \setminus S$ has a sign $g(H)= \pm 1$,
  and for each vertex $v$ in such an $H$, $f(v)$ is $g(H)$ times the graph distance between $v$ and $S$.
  An integer-valued Lipschitz function on $G$ may be regarded as a homomorphism from $G$ to a suitably long
  path with a loop at each vertex, and $\spr(G)$ measures how widely distributed along the path we can make the images
  of the vertices of $G$.
  
It is easy to see that the complete graph $\Kn$ has spread $1/4$
if $n$ is even and $1/4 - 1/(4n^2)$ if $n$ is odd. This of course
gives the minimum possible values of the spread for graphs of order $n$.
The maximum is $(n^2-1)/12$, attained by the path $\Pn$.

Denote the graph distance between vertices $u$ and $v$ by $d_G(u,v)$,
  and let ${\rm diameter}(G)$ be the maximum value of $d_G(u,v)$.
  It is easily seen from (\ref{eqn.var}) that
 $\spr(G) \leq \frac14 \ {\rm diameter}(G)^2$, and similarly that
  $\spr(G) \leq \frac1{2n^2} \sum_{v,w} d_G(v,w)^2$, so the spread is at
  most half the mean squared distance between vertices. 
  Our results will imply that the spread is typically much smaller
for random graphs. 

  Given a list of graphs $G_1,\ldots,G_d$ the \emph{Cartesian product}
  $\prod_i G_i$ is the graph with vertex set
  $\prod_i V(G_i)$, in which two vertices $(u_1,\ldots,u_d)$ and $(v_1,\ldots, v_d)$ are adjacent if and only if
  they differ in exactly one co-ordinate $i$ and $u_i$ and $v_i$ are adjacent in $G_i$.
  It is implicit in~\cite{alon1998aii} and explicit in~\cite{bht2000} that, assuming the $G_i$ are connected,
$$\spr (\prod_i G_i) = \sum_i \spr(G_i).$$
  For example, the hypercube $\Qx^d$ is $\Kx_2^d$ (the product of $d$ copies of $\Kx_2$);
  and since $\spr(\Kx_2)=1/4$ we see that $\spr(\Qx^d)=d/4$.

  Alon, Boppana and Spencer in~\cite{alon1998aii} considered the case of a fixed connected graph $G$, 
  and were interested in tight isoperimetric inequalities concerning $G^d$ for large $d$.
  Given a graph $H$, a set $S$ of vertices of $H$ and $t>0$,
  let $B(S,t)$ denote the set of vertices at distance at most $t$ from $S$ (the $t$-{\em ball around} $S$);
  and let
$$ g(H,t) = \max_{|S| \geq |V(H)|/2} \frac{|V(H) \setminus B(S,t)|}{|V(H)|}$$
  where the maximum is over subsets $S$ of at least half the vertices of $H$.
  Thus $g(H,t)$ is the maximum proportion of vertices at distance $>t$ from a set of at least half the vertices.
  From Theorem 1.1 in~\cite{alon1998aii} (which gives a more general result) we have
\begin{theorem} \label{thm.abs} \kolla
  Let $G$ be a connected graph and let $\gam=\spr(G)$. 
Then for $d^{\frac12} \ll t \ll d$
$$ 
g(G^d,t) = e^{- \frac{t^2}{2d\gam} (1+o(1))} \qquad \mbox{as } d \to
\infty.
$$
\end{theorem}

\subsection{Our results on the spread of random graphs}

  We use \whp{} (\emph{with high probability}) for events with probability $1-o(1)$ as \ntoo.
  Our focus is on whether or not the spread is bounded \whp{} in various models of sparse random graph.
  In these models typical degrees are small and \whp{} the mean path length is $\Theta(\log n)$
  (and so the mean squared path length is $\Omega(\log^2n)$): see for
  example Durrett~\cite{Durrett-2007} \kolla
  for results on diameter and mean path length, and the discussion at the
  end of \refS{Sgnp}. (We use $\log$ to denote natural logarithm, though often the base is
  irrelevant, as in $O(\log n)$.)
  \smallskip
  
  We start with random regular graphs $G(n,d)$ with fixed degree $d\ge3$, as that is the easiest case.
  It is well known that 
 \whp\ \gnd\ is connected \cite{Bollobas1981} and so we may talk of $\spr(\gnd)$.
  In \refS{Sregular} we show:
 %
\begin{theorem} \label{C1}
  There exists a constant $\CC\CCdef\CCti$ such that, for every fixed $d\ge3$,
\whp{} $\spr(\gnd)\le\CCx$.
\end{theorem}
 \noindent
 In fact we prove a stronger result, Theorem~\ref{T1}, giving an exponential tail inequality for Lipschitz functions;
 and we derive this from a corresponding deterministic result for expander graphs, Theorem~\ref{T1exp}.
 
\smallskip

  In \refS{Sgnc} we study the random graph $\Gnc$ with fixed $c>1$, the supercritical case.
  This random graph is \whp{} disconnected,
  so we consider the spread of the largest component of $\Gnc$, which we denote by $\Hnc$.
 (Recall that for $c>1$, there is \whp{} a unique giant component $\Hnc$
  of order $\sim \gam(c)n$ for some $\gam(c)>0$.)
  It was shown in~\citep{riordan08diam} that
  there is $f(c)>0$ such that the diameter of \hnp\ is $f(c)\log n+O_p(1)$.
  However, the spread stays bounded.

\begin{theorem} \label{C2}
  For each fixed $c>1$ there exists a constant $\CC=\CCx(c)>0$ such that \whp{} $\spr(\Hnc)\le\CCx$.
\end{theorem}
  \noindent
   As with Theorem~\ref{C1} above, we actually prove a stronger result,
   Theorem~\ref{T3}, 
   giving a tail inequality for Lipschitz functions which is exponential in
   $\sqrt{n}$. 
\smallskip

  In \refS{Sgnp} we study the random graph $\Gnc$ in the barely supercritical case
  when $c=1+\eps$,
  and show that the spread tends to infinity (in probability)
  at least at the rate~$\eps\qww$.
\begin{theorem}\label{behaves}
  Let $p=(1+\eps)/n$ with $\eps = \eps(n) \to 0$ and $\eps^3 n \to \infty$ as $n \to \infty$.
  Then \whp\ the giant component $H_{n,p}$ of $\gnp$ satisfies $\; \spr(H_{n,p}) = \Omega(1/\eps^2)$.
\end{theorem}
\noindent
  We do not have matching upper bounds here,
  but there are precise results on the diameter 
  which yield upper bounds that complement (though do not quite match) the lower bound. 
  For $\eps$ as here,
  \whp{} the diameter of $H_{n,(1+\eps)/n}$ is $(3+o(1))\eps^{-1} \log(\eps^3 n)$,
  see Ding, Kim, Lubetsky and Peres~\cite{DKLP10} and Riordan and Wormald~\cite{riordan08diam};
  and it follows immediately that 
  \whp{} $\spr(H_{n,p}) = O(\eps^{-2}\log^2(\eps^3 n))$,
  which is within a $\log^2(\eps^3 n)$ factor of the lower bound in Theorem \ref{behaves}.
\smallskip

  In \refS{sec.large} we study the random regular graph $\gnd$ in the large-$d$ case.
  As noted above, the 
  spread of a connected $n$-vertex graph is
  always at least $\spr(\Kn)\ge1/4-1/(4n^2)$.  Indeed, for graphs with
  bounded average degree the spread is bounded above $1/4$ 
   -- see Proposition~\ref{lem.avdeg}.  However, we see that:
  %
\begin{theorem} \label{thm.reghighdeg}
  For each $\eps>0$ there exists a constant $d_0$ such that for each $d \geq d_0$ we have \whp\ $\spr(\gnd) < 1/4 + \eps$.
\end{theorem}

  In \refS{sec.sw} we study a basic `small world' model $R_{n,c/n}$ of a random graph, following
  Watts and Strogatz~\cite{Watts-Strogatz-1998}
  and Newman and Watts~\cite{Newman-Watts-1999}, see also for example Durrett~\cite{Durrett-2007}.
  We start with a cycle on the vertices $1,\ldots,n$ (with each $i$ and $i+1$ adjacent, where $n+1$ means 1).
  Then the other possible `short-cut' edges are added independently with probability $c/n$.

\begin{theorem} \label{thm.sw1}  
  For each fixed $c>0$ there exists a constant $\CC=\CCx(c)$ such that
  \whp{} $\spr (R_{n,c/n}) \leq \CCx$. \CCdef\CCsw
\end{theorem}
\noindent
Again,  the proof gives a stronger result with a tail
inequality for Lipschitz functions which is exponential in $\sqrt n$.
   To set \refT{thm.sw1} in context, we shall show also that for any $C$ there exists a constant $c_0=c_0(C)>0$
  such that if $0< c \leq c_0$ then \whp\ $\; \spr(R_{n,c/n}) >C$.
  \smallskip
  
  In Section~\ref{sec.lengths} we introduce edge-lengths.
  Given a connected graph $G$ with edge lengths $\ell(uv) \geq 0$, we call a real-valued function $f$ on $V(G)$ {\em Lipschitz}
  if we always have $|f(u)-f(v)| \leq \ell(uv)$.
  The {\em spread} is defined to be the maximum variance of $f(X)$ for such an $f$, where $X$ is uniform over~$V$.
\begin{theorem} \label{thm.lengths}
  There is a constant $\CC$ such that for $\Kn$ with edge lengths \iid{} uniform over ${1,\dots,n}$,
  the spread is \whp\ at most $\CCx$. \CCdef\CClengths
\end{theorem}
  \noindent
  As with other theorems, the proof in fact yields a stretched exponential tail inequality for Lipschitz functions.
The proof also shows that the same result holds for \iid{} edge lengths
  that are exponential with mean $n$, or such exponentials $+1$.

  \refT{thm.lengths} is best possible in the following sense. 
  Given any fixed $C>0$, if the edge lengths are uniform over ${1,\dots,\lceil C n \rceil}$
  then \whp{} the number of vertices with minimum incident edge-length at least $C$ is at least
  $2 \lceil n/6 \rceil$, and then the spread is at least $C^2/12$.
  (Set $f(v)=C/2$ for $\lceil n/6 \rceil$ such vertices $v$,
  $f(v)= -C/2$ for another $\lceil n/6 \rceil$ such vertices $v$, and $f(v)=0$ otherwise.)
  \smallskip

Finally \refS{Sproblems} contains some open problems arising from our work.
\smallskip

Given a graph $G$ we let $v(G)$ denote the number of vertices and $e(G)$ the
number of edges. 
We use $c_1,C_1$ etc.\ to denote various positive constants.
(We use $c_i$ for small constants and $C_i$ for large.)
In Sections \refand{Sgnc}{sec.sw}, where we consider $\Gnc$ and $\Rnc$,
these are allowed to depend 
on $c$, but they never depend on~$n$.

\begin{ack}
This work was started during the Oxford--Princeton workshop on
Combinatorics in Oxford, June 2008, continued during visits of CM
and SJ to the CRM in Montreal in September and October 2008, and finally completed following a workshop in Bellairs in March 2012.
\end{ack}


\section{Random regular graphs}  \label{Sregular}

  Recall from Section~\ref{S:intro} that \gnd{} denotes the random regular graph with degree $d$.
  (If $d$ is odd, $n$ is required to be even.)
  The following result will yield Theorem~\ref{C1} as an immediate corollary.

\begin{theorem}
  \label{T1}
Fix $d\ge3$.
There exists a constant $\cc\ccdef\ccti >0$ such that \whp{} $\gnd$ is such that
every Lipschitz function $f:\gnd\to\bbR$ satisfies
\begin{equation*}
|\set{v:|f(v)-m|\ge x}| <  2 e^{-\ccti x} n  \;\; \mbox{ for all } x\ge 0
\end{equation*}
where $m$ is a median of $f$.
\end{theorem}

In principle,  numerical values could be given for the constants
$\ccti$ and $\CCti$, but we have not tried to find explicit values,
nor to optimize the arguments. These constants can be taken
independent of $d\ge3$; in fact, it follows by  monotonicity
\cite[Theorem 9.36]{JLR} that any constant that works for
$d=3$ will work for all larger $d$ as well. We will thus consider
$d=3$ only in the proof. (Alternatively, and perhaps more elementarily,
we are convinced that the proof below easily could be modified to an
arbitrary $d$, but we have not checked the details.)

  For $\alpha>0$ we say that a graph $G$ is an $\alpha$-{\em expander} if every set $W \subset V(G)$ with
  $|W| \leq v(G)/2$ contains at least $\alpha|W|$ vertices with neighbours in $V(G)\setminus W$.
  (This is slightly at odds with the standard definition of expansion but is more convenient for our purposes.)
  Observe that an $\alpha$-expander must be connected.
  For disjoint sets $A$ and $B$ of vertices in $G$ let $E(A,B)$ be the set of edges with one end in $A$ and one in $B$;
  and let $e(A,B) = |E(A,B)|$. The \emph{Cheeger constant} of a graph $G$
  with vertex set $V$ and $e(G)>0$ is 
\begin{equation}   \label{cheeger}
  \Phi(G) = \min_{\{S\subset V:\, 0<\sum_{v \in S} d(v)\, \leq e(G)\}}
  \frac{e(S,V\setminus S)}{\sum_{v \in S} d(v)}.
\end{equation}
  $\Phi(\cdot )$ measures the edge expansion, rather than the vertex expansion, of graphs.
  We shall use the following expander property of $\gnx3$, proved (in a more general version) by \cite{BKW}
  (see also \cite{Luczak} and \cite[(proof of) Lemma 5.1]{SJ184}).

\begin{lemma}[\cite{BKW}, Lemma 5.3] \label{L1}
  There is a constant $\cc>0\ccdef\ccli$ such that \wvhp{} $\Phi(\gnx3) \geq \ccx$.
\end{lemma}

  Since $\gnx3$ has constant degree, Lemma \ref{L1} immediately implies vertex expansion for $\gnx3$,
  with the same constant. We state this as a simple lemma.

\begin{lemma} \label{L2A}
  If $G$ is regular, and $0<\ga\le\Phi(G)$, then $G$ is an $\ga$-expander.
\end{lemma}

\begin{proof}
  Let $n\=v(G)$, and let $d$ be the degree of the vertices.
Note that $G$ has precisely $dn/2$ edges.
Fix a set $W$ of vertices in $G$ with $|W| \leq n/2$.
Then $\sum_{v \in W} d(v) =d|W| \leq dn/2$, so by \eqref{cheeger}
there are at least $\Phi(G)d|W|$ edges from $W$ to its
complement. These edges have at least $\Phi(G)d|W|/d\ge\ga|W|$
endpoints in  $W$.
\end{proof}

\begin{lemma}
  \label{L2}
$\gnx3$ is \wvhp{} a $\ccx$-expander.
\end{lemma}

\begin{proof}
An immediate consequence of Lemmas \refand{L1}{L2A}.
\end{proof}

  The following deterministic result on expanders now yields Theorem~\ref{T1}.
  

\begin{theorem} \label{T1exp}
  For each $0<\alpha \leq \frac12$
  and each $\alpha$-expander $G_n$ on $[n]$, every Lipschitz function $f$ for $G_n$ satisfies
\begin{equation*}
  |\set{v:|f(v)-m|\ge x}| <  2 e^{-(\alpha/2) x} n  \;\; \mbox{ for all } x\ge 0
\end{equation*}
  where $m$ is a median of $f$.
\end{theorem}


\begin{proof}

  Let $f$ be a Lipschitz function on $G_n$, with median~$m$.
  We may assume that $m=0$; otherwise we replace $f$ by $f-m$.
  Let $V_t\=\set{v\in\nn:f(v)\ge t}$. Then $|V_t|\le n/2$ for $t>0$.

  If $t>0$ and $V_t$ is nonempty then
  there is a subset of $V_t$ of size at least $\alpha |V_t|$ of vertices
  $x$ that are adjacent to at least one vertex $y\notin V_t$.
  Thus $f(y)<t$, and since $f$ is Lipschitz, we have $f(x)<t+1$ for every such $x$.
  Consequently, $|V_{t+1}|\le(1-\alpha)|V_t|$ when $t>0$.
  Since $|V_1|\le n/2\le(1-\alpha)n$, 
  we obtain by induction, for simplicity considering integers $k$ only,
\begin{equation*}\label{a4}
  |V_{k}|\le (1-\alpha)^kn \leq e^{-\alpha k} n, \qquad k=1,2,\dots
\end{equation*}

  By symmetry, we have the same estimate for $\set{v:f(v)\le -k}$, and thus, for every $x\ge1$,
\begin{equation*}
  |\set{v:|f(v)|\ge x}| \le  2e^{-\alpha \floor x} n <  2e^{-(\alpha/2) x} n.
\end{equation*}
  Finally, since $2e^{-\alpha/2} > 1$ the last bound also holds for each $0 \leq x <1$,
  which completes the proof.
\end{proof}


\section{$G_{n,c/n}$ with $c > 1$ fixed.}
\label{Sgnc}
  In this section we consider supercritical random graphs, and prove the following theorem,
  which immediately implies Theorem~\ref{C2}. 
\begin{theorem}\label{T3}
  Given fixed $c>1$ there is a constant $\cc=\ccx(c)\ccdef\ccb$
such that \whp{} the giant component $H= H_{n,c/n}$ of\/ $G_{n,c/n}$ is such
that 
  every Lipschitz function $f$ for $H$ satisfies 
\begin{equation}\label{t3}
  |\set{v:|f(v)-m|> x}| <  2 {e^{-\ccx\sqrt{x}}} v(H) \;\; \mbox{ for all } x\ge0
\end{equation}
  where $m$ is a median of $f$.
\end{theorem}

  For this case, 
  in place of Lemma~\ref{L1} 
  we can use another result of~\cite{BKW}.
  For a graph $G$ and a set of vertices $U \subset V(G)$, we write
  $G\setminus U$ for the subgraph of $G$ induced by $V(G)\setminus U$.
  For $0 < \alpha < 1$ we say that a connected graph $H$ is an
  \emph{$\alpha$-decorated expander} if $H$ has a subgraph $F$ such that
\begin{enumerate}
  \item[(DE1)]
    $\Phi(F) \geq \alpha$;
  \item[(DE2)]
    listing the connected components of $H \setminus V(F)$ as
    $D_1,\ldots,D_\ellx$ for some~$\ellx$,
$$
\left|\left\{i: e(D_i)+ e(D_i,F) \geq x\right\}\right| \leq e^{- \alpha x} e(H);
$$
  \item[(DE3)]
    no vertex $v \in V(F)$ is adjacent to (``decorated by'') more than $1/\alpha$ of the components $D_i$.
\end{enumerate}
   Note that (DE1) implies that $F$ is connected. Note further that (DE2) implies:
\newcommand{\DEij}{(DE$2'$)}
\begin{enumerate}
  \item[\DEij]
    for all $x \geq 0$, $|\{i:v(D_i) \geq x\}| \leq e^{-\alpha x} e(H)$.
\end{enumerate}

  We shall use \DEij{} rather than (DE2) in what follows.
  Let us say that $H$ is a \emph{weak $\alpha$-decorated expander} if 
  (DE1), \DEij\ and (DE3) hold, and one further condition holds: 
\begin{enumerate}
  \item[(DE4)] $v(F) \geq \alpha  v(H)$.  
\end{enumerate}
  From Benjamini et.~al.~\citep{BKW} 
  (their Theorem 4.2 and Lemma 4.7, combined) we have:
\begin{lemma}
  \label{LT2}
  Fix $c > 1$. Then there is a constant $\alpha=\alpha(c)>0$ such
  that \whp{}  the giant component $H_{n,c/n}$ of $G_{n,c/n}$
  is a weak $\alpha$-decorated expander. 
\end{lemma}

Since the expansion guaranteed by \refL{LT2} is edge-expansion, we will
need to do a little work to derive the vertex expansion required to
prove Theorem \ref{T3}. The following lemma will give some
further, more elementary, properties of $\Gnc$ that suffice for
our purposes.
  Given a graph $G$ let $V_i(G)$ be the set of vertices of degree $i$, and
  let $v_i(G)=|V_i(G)|$. 

The constants $\CC\CCdef\CCedges,\CC\CCdef\CCgnc,\dots$
below may depend on $c$ and $\ga$. 

\begin{lemma}  \label{Lgnc}
  For fixed $c>1$, \Gnc{} is \whp{} such that $H=\Hnc$ satisfies the
  following properties, 
  for  suitable constants: 
\begin{romenumerate}
  \item[(P1)]
    $n':= v(H)>\gam n/2$ for some $\gam=\gam(c)>0$,
  \item[(P2)]
    $e(H)\le \CCedges v(H)$, 
  \item[(P3)]
   $v_i(H) \leq e^{-i} v(H)$ for all $i\ge\CCgnc$.
\end{romenumerate}
\end{lemma}


\begin{proof}
  It is well-known that $n'/n\pto\gam(c)>0$. 
  It is also well-known and easy to see that $e(\Gnc)/n \pto c/2$. 
These two results yield (P1) and (P2).

For (P3), 
let $d_j=d_j(\Gnc)$ be the degree of vertex $j$, and let
$X$ be the random variable $\sum_{j=1}^n e^{2d_j}$. Since each $d_j$
has a binomial $\Bin(n-1,c/n)$ distribution,
\begin{equation}
  \E X = n \E e^{2d_1} = n\Bigpar{1+\frac cn(e^2-1)}^{n-1} \le n e^{c(e^2-1)}.
\end{equation}
  A similar calculation shows that
$$\Var X = n \Var(e^{2 d_1})+ n(n-1) \Cov(e^{2 d_1},e^{2 d_2}) = O(n).$$
Consequently, by Chebyshev's inequality, \whp
\begin{equation*}
 \sum_{i=0}^\infty e^{2i} v_i(\Gnc) = X \le e^{c e^2} n.
\end{equation*}
The result follows, using (P1).
\end{proof}

\begin{remark}\label{RP3}
  The proof of (P3) shows that it could be strengthened to
   $v_i(H) \leq e^{-Ci} v(H)$ for all $i\ge\CCgnc$, for any fixed $C$;
conversely, it would for our purposes be enough that
$v_i(H) \leq 2e^{-\ga i} v(H)$ for all $i$. For simplicity, we use the
version above.
\end{remark}

  Let us call a connected graph $H$ a 
  \emph{well-behaved weak $\alpha$-decorated expander} if it is a weak
  $\alpha$-decorated expander and it has properties  (P2) and (P3) in the
above lemma for some constants $\CCedges,\CCgnc$, where we for definiteness 
assume $\CCedges=\CCgnc=\ga\qw$.
  By \refL{Lgnc},  \refL{LT2} can be improved (possibly reducing $\ga$):
\begin{lemma}
  \label{LT2+}
  Fix $c > 1$. Then there is a constant $\alpha=\alpha(c)>0$ such
  that \whp{}  the giant component $H_{n,c/n}$ of $G_{n,c/n}$
  is a well-behaved weak $\alpha$-decorated expander. 
\nopf
\end{lemma}

  \refT{T3} now follows immediately from 
  the following deterministic lemma.

\begin{lemma} \label{lem.decexp}
  Let the connected graph $H$
  be a well-behaved weak $\alpha$-decorated expander. 
  Then \eqref{t3} holds for every Lipschitz function $f$ on $H$, for some
  $\ccb$ depending on $\ga$.
\end{lemma}

\begin{proof}
  Fix a subgraph $F$ of $H$ 
  which verifies that $H$ is a weak $\alpha$-decorated expander. 
  Let $D$ be the graph $H \setminus V(F)$, and let $D_1,\ldots,D_\ellx$ be the components of $D$.
  Fix a Lipschitz function $f$ on $H$. Let $n' = v(H)$ as in Lemma \ref{Lgnc}.

We write $H_{\ge t}$ for the set of vertices $v \in V(H)$ with $f(v)
\geq t$ and define $H_{>t}$, $H_{\le t}$, $H_{<t}$ similarly.
Additionally, we write $F_{\ge t}$ (and $F_{> t}$ et cetera) for $V(F)
\cap H_{\ge t}$, and
$D_{\ge t}$ (et cetera) for $H_{\ge t} \cap V(D)= H_{\ge t} \setminus V(F)$.
We also assume as in the proof of \refT{T1} that $f$ has median $m=0$;
hence $|H_{\le0}|,|H_{\ge0}|\ge n'/2$.

Our plan of attack is as follows. First, we find a large subset of $V(F)$
consisting exclusively of vertices $v$ with $f(v)$ bounded above by a
constant. 
Such a set is not quite guaranteed by the fact that $|H_{\le 0}| \geq n'/2$,
because $H_{\le 0}$ may be largely contained within $V(H)\setminus V(F)$. 
However, we shall use properties \DEij{} and (DE3) to find such a set. Second, we use the
expansion of $F$ to 
show that the sets $F_{\ge t}$ decay rapidly in size as $t$ grows. Finally, we use the fact that the decorations $D_i$ are typically small
and do not attach to very many vertices of $F_{\ge t}$, to show that the sets $D_{\ge t}$ also decay rapidly in size as $t$ grows. We now turn
to the details.
For simplicity we prove the theorem for $x$ integer, which easily
implies the more general statement.

  For $\lambda > 0$ let $F^{\lambda}$ be the union of $F$ and all components $D_i$ with $v(D_i) < \lambda$.
  By property (P2), 
  $e(H) \leq \CCedges n'$.   
  By property \DEij, for any $\lambda > 0$ we have
\begin{equation}\label{t2a}
  \begin{split}
\sum_{\{i:~v(D_i) \geq \lambda\}} v(D_i)
&=
\sum_{j=0}^{\infty} \sum_{\{i:~2^j\lambda \leq v(D_i) < 2^{j+1}\lambda\}} v(D_i) \\
& \leq
\sum_{j=0}^{\infty} 2^{j+1}\lambda e^{-\alpha\lambda 2^j} \cdot \CCedges n'.
  \end{split}
\end{equation}

  Choose $\gl=\lambda_1$ large enough that the upper bound in (\ref{t2a}) is less than 
  $n'/4$; 
  then $F^{\lambda_1}$ contains at least $3n'/4$ vertices.
Since at most $n'/2$ vertices $v$ in $H$ have $f(v)>0$,
it follows that at least $n'/4$ of the vertices in $F^{\lambda_1}$
have $f(v) \leq 0$. 
Since each component of $F^{\lambda_1}\setminus F$ has less than
$\lambda_1$ vertices,
either $|F_{\le 0}| \geq n'/8$ or at least $n'/(8\lambda_1)$
components of $F^{\lambda_1}\setminus F$ contain a vertex
of $H_{\le 0}$. Since all vertices in $F^{\lambda_1}\setminus F$ have
distance at most $\lambda_1$ from $F$,
property (DE3) 
and the Lipschitz property of $f$
then guarantee that in either case (assuming $\gl_1>\ga$ as we may)
$$
|F_{\le \lambda_1}|
\geq \frac{\alpha n'}{8\lambda_1}
=:\cc n'.
$$

  Since every vertex of $F$ has at least one neighbour in $F$, it follows that
  $\sum_{v \in F_{\le \lambda_1}} d_F(v) \geq \ccx n'$.
  Assuming that $\sum_{v \in F_{\le\lambda_1}} d_F(v) \leq e(F)$,
  by the expansion property (DE1) we thus have
  that $e(F_{\le \lambda_1},F_{> \lambda_1}) \geq \alpha \ccx n'$.
  The Lipschitz property of $f$ implies that each edge
  in $E(F_{\le \lambda_1},F_{> \lambda_1})$ has one endpoint in
  $F_{\le\lambda_1+1}\setminus F_{\le \gl_1}$, and thus
\begin{equation*}
\sum_{v \in F_{\le\lambda_1+1}\setminus F_{\le \gl_1}} d_F(v)
\ge
e(F_{\le \lambda_1},F_{> \lambda_1}) \geq \alpha \ccx n'.
\end{equation*}
  Repeatedly applying property~(DE1) 
  in this manner, and using property~(P2), we see that \whp{}
  $\sum_{v \in F_{\le \lambda_2}} d_F(v)> e(F)$, where we may take
  $\lambda_2 = \lambda_1+ \CCedges/(\alpha \ccx) +1$.

  We next apply the expansion of $F$  and properties (P2)--(P3) 
  to bound the sizes of sets $F_{>\lambda_2+i}$ for positive integers $i$.
  As $i$ becomes large and the sets $F_{>\lambda_2+i}$ become small, the
  proportion of the sum $\sum_{v \in F_{> \lambda_2+i}} d_F(v)$ due to vertices of large degree
  may increase; this is the reason we are only able to show that the sizes of
  the sets $F_{> \lambda_2+i}$ decay exponentially quickly in $\sqrt{i}$.

  For given $x>0$, let $a_x$ be the smallest integer $\ge\CCgnc$ such that
  $\sum_{i>a_x}^{\infty} ie^{-i} \leq \alpha x/2$.  Since
  $\sum_{i>a}^{\infty} ie^{-i} \leq \sum_{i>a}^{\infty} e^{-i/2} \leq 3e^{-a/2}$,
  there exists $\CC$ large enough that $a_x \leq \CCx \log (1/x)$ for all $x\le 1/2$.

  For $\lambda \geq \lambda_2$, if $t'=\sum_{v \in F_{> \lambda}} d_F(v)$ then
  $t'\leq e(F)$ by our choice of $\lambda_2$.  For 
  $0 \leq t \leq t'$, we thus have $e(F_{> \lambda},F_{\le \lambda}) \geq \alpha t'\ge \ga t$ by (DE1).
\newcommand\dfl{\partial F_{>\gl}}
  Let
$$\dfl=\set{v\in F_{> \lambda}: v \text{ has a neighbour in }F_{\le \lambda}}.
$$
  Then for any $t$ as above,
  $\sum_{v \in \dfl} d_F(v) \geq e(F_{> \lambda},F_{\le \lambda}) \geq\alpha t$.
  Also, applying property~(P3) and using the definition of $a_{t/n'}$, 
$$\sum_{v \in F} d_F(v)\ett{d_F(v) > a_{t/n'}} \leq \sum_{i>a_{t/n'}}^{\infty} ie^{-i} \cdot n' \leq \alpha t /2$$
  and so
$$\sum_{v \in \dfl} d_F(v)\ett{d_F(v) \leq a_{t/n'}} \geq \alpha t - \sum_{v \in F} d_F(v)\ett{d_F(v) > a_{t/n'}}
 \geq \alpha t/2.$$
  Hence. assuming also that $t \leq n'/2$,
  \begin{equation}\label{mr}
\lrabs{\dfl}
\geq\frac{\ga t/2}{a_{t/n'}}
\geq \frac{\ga t}{2\CCx \log (n'/t)}
:= \cc\frac{t}{\log (n'/t)}.
  \end{equation}
  Now fix $\lambda\geq \lambda_2$. Taking $t=|F_{> \lambda}| \leq \sum_{v \in F_{> \lambda}} d_F(v) = t'$,
  we also have $t \leq |H_{> 0}| \leq n'/2$, so 
\eqref{mr} applies with this choice of $t$.
  Furthermore, the Lipschitz property of $f$ implies that $\dfl\subseteq F_{\le\gl+1}$, and so
\begin{equation*}
  |F_{> \lambda+1}|\le|F_{> \lambda}|-|\dfl| \leq t(1-\ccx/\log(n'/t)).
\end{equation*}
  Next, for integers $i \geq 1$, let $k_i = \ceil{i/\ccx}$.
  Then for all  $t \geq n'/2^i$, we have $(1-\ccx/\log(n'/t))^{k_i}<1/2$.
It follows immediately that for all integers $i\ge1$ we have
$$
  |F_{> \lambda_2 + \sum_{j=2}^i k_j}| \leq \frac{n'}{2^{i}},
$$
  so there is $\CC > 0$ such that for all real $x \geq 1$,
and trivially for $0\le x\le1$,
\begin{equation}\label{t2b}
  |F_{>\CCx x^2}| \leq \frac{n'}{2^x}.
\end{equation}

  We now deal with the elements of the `decorations' graph $D$, 
  and assume that its components $D_1,\ldots,D_\ellx$ are listed so that $v(D_1) \geq \cdots \geq v(D_\ellx)$.
  We first remark that by \DEij{} and (P2), 
  if $m_k$ is the number of components $D_i$ of $D$ with $v(D_i) \geq k$, 
  then $m_k\le \CCedges n' e^{-\alpha k}$ for all integers $k \geq 1$.
  Hence, for any real $t$ with $0 < t \le n'$, we have, with $x=\log(\CCedges n'/t)/\ga$,
\begin{equation}\label{t2c}
  \begin{split}
\sum_{j=1}^{\floor t} v(D_j)
& =
 \sum_{k=1}^{\infty} \min\bigpar{\floor t,m_k}
\le
\sum_{k=1}^{\infty} \min\bigpar{t,\CCedges n' e^{-\ga k}}
\\&
=
\sum_{k\le x} t+\sum_{k> x} \CCedges n' e^{-\ga k}
\le
\CC t(\log n' + 1 -\log t).
  \end{split}
\end{equation}

  Next, for $w \in V(D)$, let $D^w$  
  be the component of $D$ containing $w$ and fix 
  an arbitrary vertex $u^w$ of $F$ that is decorated with $D^w$.
  By (DE3), for any set $S \subseteq V(F)$ with $|S|\le s$,
  the total number of components that decorate $S$ is at most $s/\alpha$.
  It then follows from \eqref{t2c} that
\begin{equation}\label{t2d}
  |\{w\in V(D)~:~u^w \in S\}|
 \leq  \sum_{j=1}^{\lfloor s/\alpha\rfloor} v(D_j)
\le \CC s(\log n' + 1 -\log s)
\CCdef\CCttd
\end{equation}
  if $s\le \ga n'$, and by taking $\CCttd \geq 1/\alpha$ we see that the
  inequality in fact holds for all $s \leq n'$.
  For $i \geq 0$, 
  if $w \in D_{>i}$ then one of the following two events must occur.
\begin{itemize}
\item[(a)] $v(D^w) \geq 3i/4$.
\item[(b)] $d(w,u^w)<3i/4$ and then $u^w \in F_{> i/4}$. 
\end{itemize}
  By \DEij{} and (P2), 
\begin{equation}\label{3i4}
  |\{w \in D: v(D^w) \geq 3i/4\}|
  \leq \sum_{j \geq 3i/4} j \cdot \CCedges n' e^{-\alpha j} 
\leq \CC n' (i+1) e^{-3\alpha i/4}.
\end{equation}
  Furthermore, by (\ref{t2b}),
$$
  |F_{> i/4}| \leq n'/2^{\cc\sqrt{i}}
$$
  and thus by (\ref{t2d}) we have
$$
  |\{w \in D:u^w \in F_{>i/4}\}|
  \leq \CCttd \frac{n'}{2^{\ccx \sqrt{i}}}\left(1+\ccx\sqrt{i}\log 2\right),
$$
  so for all $i$ we have 
\begin{equation}\label{i4}
  |\{w \in D:u^w \in F_{>i/4}\}| \leq \CC{n'}{e^{-\cc\sqrt{i}}}
\end{equation}
  for suitable constants $\CCx$ and $\ccx>0$.
  Thus, by \eqref{3i4} and \eqref{i4},
\begin{equation}
  |D_{>i}| \le
  |\{w \in D: v(D^w) \geq 3i/4\}| + |\{w \in D: u^w \in F_{> i/4}\}|
  \leq \CC n'
e^{-\ccx\sqrt{i}}.
\end{equation}
  Hence, using this together with (\ref{t2b}) to bound $|F_{>i}|$,
  we have
\begin{align*}
|H_{>i}|     = |F_{>i}|+|D_{>i}|
         \leq  \CC e^{-\cc\sqrt{i}} n' 
\CCdef\CCy \ccdef\ccy
\end{align*}
  for fixed $\CCy$ sufficiently large.
  Now note that $-f$ is also a Lipschitz function on $H$ with a median $0$,
  and so for all $i \geq 0$ 
$$ | \{ v:|f(v)| >i\}| \leq 2   \CCy e^{-\ccy\sqrt{i}} n' .$$
  To complete the proof, let $i_0 >0$ satisfy 
$2 \CCy e^{-\ccy\sqrt{i_0}} \leq 1$;
  and then choose $\cc$ with $0<\ccx \leq \ccy$ satisfying $2
  e^{-\ccx\sqrt{i_0}} > 1$.
  Now $2 e^{-\ccx\sqrt{i}}> \min \{1,2 \CCy e^{-\ccy\sqrt{i}}\}$ for each $i \geq 0$; and so
  $ | \{ v:|f(v)| >i\}| < 2 e^{-\ccx\sqrt{i}} n' $ for all $i \geq 0$, and the theorem follows.
\end{proof}


\section{$G_{n,(1+\eps)/n}$ with $\eps \rightarrow 0$,
  $\eps \gg n^{-1/3}$.}
\label{Sgnp}

  In this section we consider the barely supercritical case, and prove Theorem~\ref{behaves}.
Fix a function $\eps=\eps(n)$ as above and let
$p=(1+\eps)/n$. As above, denote by $H_{n,p}$ the largest
component
of $G_{n,p}$. Additionally, write $C_{n,p}$ (resp.~$K_{n,p}$) for the
core (resp.~kernel) of $H_{n,p}$.
For such $\eps$, it is known (see \cite{luczak91cycles} and also
\cite{JLR}, Chapter 5) that \whp{}
\begin{align}
  v(H_{n,p}) & = (1+o(1)) 2\eps n, \nonumber\\
  v(C_{n,p}) & = (1+o(1)) 2 \eps^2 n,\mbox{ and} \label{sizes}\\
  v(K_{n,p}) & = (1+o(1)) \frac{4}{3} \eps^3 n. \nonumber
\end{align}
For a connected graph $G$, we write $\kappa(G) = e(G)-v(G)$, and
call $\kappa$ the \emph{excess} of $G$.
A moment's reflection reveals that $\kappa(H_{n,p}) = \kappa(C_{n,p})
= \kappa(K_{n,p})$, and
it is known (\cite{janson93birth, JLR, luczak90component}) that for
$\eps$ as above, \whp{}
\begin{align}
\kappa(H_{n,p}) & = (1+o(1)) \frac{2}{3}\eps^3 n. \label{excess}
\end{align}
We fix $\delta < 1/10$ and say $H_{n,p}$ \emph{behaves} if
$$
  (2-\delta) \eps n \leq v(H_{n,p}) \leq (2+\delta) \eps n,
$$
  and if similar inequalities hold for $v(C_{n,p})$, $v(K_{n,p})$, and $\kappa(H_{n,p})$.
  By the above comments, \whp{} $H_{n,p}$ behaves. Using this fact and one further lemma, we
  may prove Theorem~\ref{behaves}.

The complement $\hnp\setminus V(\cnp)$ of the core in the largest
component $\hnp$ is a forest consisting of trees that are attached to
the core by (exactly) one edge each. We call these trees
\emph{pendant}, and denote them (in some order) by $T_1,\dots,T_N$.
We begin with an estimate of the maximum size of the pendant trees.

\begin{lemma}\label{Lmaxtree}
There exists  a constant $\CC$ such that \whp
\begin{equation}
\label{pavlov}
  \max_i v(T_i) \le\CCx \eps\qww\log(n\eps^3).
\end{equation}
\CCdef{\CCpavlov}
\end{lemma}

\begin{proof}
  We create another forest by removing all edges in the core
  $\cnp$ from $\hnp$; the result is a forest where each component
  consists of a single vertex in $V(\cnp)$ together with all pendant
  trees attached to it (if any). We regard these trees as rooted, with
  the vertices in $V(\cnp)$ as the roots, and denote them by
  $T^*_w$,  $w \in V(\cnp)$.

  Conditioned on $V(\hnp)$ and $\cnp$, this forest $\set{T^*_w}_{w}$
  is a uniformly distributed forest of rooted trees, with given sets of
  $M:= v(\cnp)$ roots and $m:= v(\hnp)-M$ non-roots.

  The maximum size of a tree in a random forest of rooted trees has
  been studied by Pavlov \cite{Pavlov77} (see also \cite[Section
  3.6]{kolchin86mappings} and \cite{Pavlov}). In our case we have, if $\hnp$
  behaves and $n$ is large enough,
$(2-\gd)n\eps^2 \le M \le (2+\gd)n\eps^2$
and 
$(2-2\gd)n\eps \le m \le (2+\gd)n\eps$. In particular, $m/M\to\infty$
  and $m/M^2 \le (n\eps^3)\qw\to0$. This is the range of
  \cite[Theorem 3 (and the remark following it)]{Pavlov77}, which
  implies that \whp, conditioned on $M$ and $m$,
\begin{equation*}
  \max_w v(T^*_w)=(1+o(1))\ \frac{2m^2}{M^2} \log\parfrac{M^2}{m}
\le\CCx \eps\qww\log(n\eps^3).
\end{equation*}
  The same estimate thus holds unconditionally \whp, and the result
  follows since every pendant tree is contained in some $T^*_w$.
\end{proof}

\begin{proof}[Proof of Theorem \ref{behaves}]
Since $H_{n,p}$ behaves \whp{,} it suffices to prove that \emph{given}
that $H_{n,p}$ behaves, \whp{} $\spr(H_{n,p}) =
\Omega(1/\eps^2)$.
We shall define a Lipschitz function $f$ on the vertices of $H_{n,p}$ for which,
given that $H_{n,p}$ behaves, \whp{} $\Var(f) \geq \gamma/\eps^2$ for some fixed $\gamma > 0$.
We define $f$ in a few steps, starting from the core. We say that $e \in E(K_{n,p})$ has \emph{length} $\ell(e)$ if the path in $C_{n,p}$ corresponding to $e$ contains
$\ell(e)$ edges (so $\ell(e)-1$ internal vertices). Since $H_{n,p}$ behaves,
\begin{align}
  e(K_{n,p})    & = v(K_{n,p}) + \kappa(K_{n,p}) \nonumber\\
            & \leq \Bigpar{\frac{4}{3} + \delta}\eps^3 n +
\Bigpar{\frac{2}{3} + \delta}\eps^3 n \nonumber\\
            & = (2+2\delta)\eps^3 n,\label{eknp}
\end{align}
and
\begin{align}
|V(C_{n,p})\setminus V(K_{n,p})|    & \geq (2-\delta)\eps^2 n - \Bigpar{\frac{4}{3} + \delta}\eps^3 n \nonumber\\
                            & \geq (2-2\delta)\eps^2 n,\label{cminusk}
\end{align}
for $n$ sufficiently large.

We say that an edge $e \in E(K_{n,p})$ is \emph{short} if
$$
\ell(e) \leq \left\lfloor \frac{1-\delta}{2\eps(1+\delta)}\right\rfloor
$$
(and is long otherwise), and that $v \in V(C_{n,p})\setminus V(K_{n,p})$ is \emph{useless} if it is contained in a path corresponding to a short edge (and is useful otherwise).
By (\ref{eknp}) and (\ref{cminusk}), the number of useful vertices is at least
\begin{align}
  |V(C_{n,p})\setminus V(K_{n,p})|-e(K_{n,p}) \cdot \frac{1-\delta}{2\eps(1+\delta)} \geq (1-\delta)\eps^2 n. \label{usefulsize}
\end{align}
Next, let $r=r(n)$ be the largest integer divisible by 3 and with $2r \leq (1-\delta)/(2\eps(1+\delta))$. For each long edge $e \in E(K_{n,p})$, let $P_e$ be the
path in $C_{n,p}$ corresponding to $e$ (so the endpoints of $P_e$ are in $K_{n,p}$), and let $P_e'$ be a sub-path of $P_e$, not containing
the endpoints of $P_e$, which is as long as possible subject to the condition that $2r$ divides
  $v(P_e')$ (picked according to some rule); such a sub-path certainly exists since
$$
  v(P_e) =e(P_e)+1 \geq \left\lfloor \frac{1-\delta}{2\eps(1+\delta)}\right\rfloor + 2 \geq 2r+2,
$$
  so $P_e$ has at least $2r$ internal vertices. Since $P_e$ has $v(P_e)-2$ internal vertices, we also have that
  $v(P_e') \geq (v(P_e)-2)/2$, so by (\ref{eknp}) and (\ref{usefulsize}),
\begin{align}\label{subpaths}
|\{v:v \in P_e'\mbox{ for some } e \in E(K_{n,p})\}|    & \geq \sum_{\{e:e~\mathrm{is~long}\}} \frac{v(P_e)-2}{2} \nonumber\\
                                        & \geq \frac{(1-\delta) \eps^2 n}{2} - 2(1+\delta)\eps^3 n \nonumber\\
                                        & \geq \frac{(1-\delta)\eps^2 n}{3},
\end{align}
for $n$ large enough.
We now define the restriction of $f$ to $V(C_{n,p})$ as follows.
\begin{itemize}
\item If $v \in V(K_{n,p})$, $v$ is useless, or $v$ is not in $P_e'$
  for any long edge $e$, then set $f(v)=0$.
\item For each long edge $e$, repeat the sequence of values $12\ldots
  (r-1)rr(r-1)\ldots 1$ along $P_e'$
(so if $v$ is the $i$'th or $(2r+1-i)$'th vertex $\mod 2r$ along some
  path $P_e'$ then $f(v)=i$).
\end{itemize}
To extend $f$ from $C_{n,p}$ to the remainder of $H_{n,p}$, for each
vertex $v \in V(H_{n,p})$, we define
the \emph{point of attachment} $a(v)$ to be the vertex $x \in C_{n,p}$
whose distance from $v$ in $H_{n,p}$ is minimum, and we set
$f(v)=f(a(v))$. 
In other words, for each pendant tree $T$ in $H_{n,p}$ that hooks up
to the core at $v \in V(C_{n,p})$, we set $f(w)=f(v)$ for all $w \in
V(T)$.

To analyze the variance of $f$, for $i=1,2,3$, let
$$
B_i = \left\{v \in V(C_{n,p})~:~ \frac{i-1}{3}r < f(v) \leq \frac{i}{3}r\right\},
$$
and let $B_0$ be all remaining vertices of $C_{n,p}$, i.e., those
with $f(v)=0$.
By the definition of $f$ and since $3$ divides $r$, the sizes of $B_1,B_2$, and $B_3$ are identical, and are each at least $(1-\delta) \eps^2 n/9$.
Also, for $i=1,2,3$, let $B_i^+$ be the set of vertices $v \in V(H_{n,p})$ with $a(v) \in B_i$.
We will prove the following assertion:
\begin{itemize}
\item[($\star$)] given that $H_{n,p}$ behaves, \wvhp{} $|B_i^+| \geq
  \eps n/44$ for each $i=1,2,3$. 
\end{itemize}
Assuming for the moment that ($\star$) holds, we can quickly complete the
proof of the theorem. 
For each graph $\hnp$ which behaves, the corresponding
(fixed) function $f$ satisfies 
\begin{align*}
  \Var(f) &= \frac1{2(n')^{2}} \sum_{x,y \in V(\hnp)}(f(x)-f(y))^2\\
& \geq  (n')^{-2} \sum_{x \in B_1^+} \sum_{y \in B_3^+}(f(x)-f(y))^2\\
& \geq  (n')^{-2} |B_1^+||B_3^+| \ r^2/9\\
& \geq
\frac{(\eps n/44)^2}{((2+\delta)\eps n)^2}\frac{r^2}{9}\\
&= \frac{r^2}{69696 (1+\delta/2)^2}.
\end{align*}
But $r=\Omega(1/\eps)$, and so it follows that, conditional on the
event that $\hnp$ behaves, \whp\ $\Var(f) =
\Omega(\eps^{-2})$,
as needed.

It thus remains to prove ($\star$), and we now turn to this.
Let $X=|B_1^+|$, the number of vertices $v \in V(H_{n,p})$ with $a(v) \in B_1$. Our aim is to show that $\P\{X \geq \eps n/44\} = 1-o(1)$.

We note that given $C_{n,p}$, we can specify $H_{n,p}$ by listing the pendant subtrees of $H_{n,p}$, and their points of attachment in $C_{n,p}$, as
$T_1,\ldots,T_N$ and $U_1,\ldots,U_N$.
By routine calculation it is easily seen that
given $C_{n,p}$ and the pendant subtrees $T_1,\ldots,T_N$, the
points of attachment $U_1,\ldots,U_N$ are independent and uniformly random elements of $V(C_{n,p})$.
We further note that given $C_{n,p}$ and the pendant subtrees $T_1,\ldots,T_N$, we can determine whether or not $H_{n,p}$ behaves.
Then, recalling \refL{Lmaxtree}, 
\begin{align}
\P\{X \geq \eps n/44\}  & \geq \inf_{\cS} \P\{X \geq \eps n/44~|~C_{n,p},T_1,\ldots,T_N\}-o(1),
\end{align}
where $\cS$ represents all possible choices of $C_{n,p}$ and $N$ and
$T_1,\ldots,T_N$ for which $H_{n,p}$ behaves and
\eqref{pavlov} holds. 
  Fix any such choice and let $t_i = v(T_i)$ for $i=1,\ldots,N$. To
shorten coming formulae, let
$$
\Pc{\cdot} = \P\{\cdot~\vert~C_{n,p},T_1,\ldots,T_N\},
$$
and define $\Ec$ and $\Varc$ similarly. 
Given $C_{n,p}$ and $T_1,\ldots,T_N$, we may write $X$ as
$$
X = |B_1| + \sum_{i=1}^N t_i \ett{U_i \in B_1}.
$$
  Since $H_{n,p}$ behaves, by the estimates above,
$$
  \frac{|B_1|}{v(C_{n,p})} \geq \frac{(1-\delta)\eps^2
  n/9}{(2+\delta)\eps^2 n}
\ge \frac{1-\delta}{18(1+\delta)} \geq \frac{1}{22}.
$$
Since the points of attachment $U_1,\ldots,U_N$ of $T_1,\ldots,T_N$ in $C_{n,p}$ are uniform and $\sum_{i=1}^N t_i = |V(H_{n,p})\setminus V(C_{n,p})|$,
it thus follows that 
\begin{equation}
\label{klm}
  \Ec(X) =
  |B_1| + \frac{|B_1|}{v(C_{n,p})}\cdot|V(H_{n,p})\setminus V(C_{n,p})|
> \frac{\eps n}{22},
\end{equation}
the preceding inequality holding for $n$ sufficiently large since
$H_{n,p}$ behaves.
Next, given $C_{n,p}$ and $T_1,\ldots,T_N$,
$|B_1|$ is determined and 
$X-|B_1|$ is a sum of independent random variables
$t_i\ett{U_i \in B_1}$, $i=1,\ldots,N$.
Hence, 
\begin{equation*}
  \Varc(X)=
\sum_{i=1}^N t_i^2 \Varc(\ett{U_i \in B_1})
\le
\sum_{i=1}^N t_i^2.
\end{equation*}
By Chebyshev's inequality, when $n$ is large enough that \eqref{klm} holds,
we thus have
\begin{align*}
\Pc{X < \frac{\eps n}{44}}
 & \leq \frac{\sum_{i=1}^N t_i^2}{(\eps n/44)^2}.
\end{align*}
Since we have assumed that \eqref{pavlov} holds, and that $H_{n,p}$ behaves,
\begin{equation*}
  \sum_{i=1}^N t_i^2 \le   \max_{1\le i\le N} t_i\cdot \sum_{i=1}^N t_i
\le\CC\eps\qww\log(n\eps^3)\cdot n\eps
\end{equation*}
and thus, for $n$ large enough,
$$
\Pc{X < \frac{\eps n}{44}}
\leq \CC \frac{n\eps\qw\log(n\eps^3)}{(\eps n)^2 }
=\CCx \frac{\log(n\eps^3)}{n\eps^3 }
 \rightarrow 0
$$
as $n \to \infty$.
An identical argument yields the same lower bound with $X$ equal to $|B_2^+|$ or $|B_3^+|$.  (We do not actually care about $|B_2^+|$.) This establishes ($\star$) and completes the proof.
\end{proof}


\section{Regular graphs with large degrees}
\label{sec.large}

We saw that \whp\ the random regular graph \gnd\ has bounded
spread for any fixed $d \geq 3$, and similarly the random graph
\Hnc \ has bounded spread for any fixed $c>1$.
  As noted in the introduction, the minimum possible values of the spread
  (achieved for the complete graph $\Kn$) are $1/4$
  if $n$ is even and $1/4 - 1/(4n^2)$ if $n$ is odd. 
This suggests another natural question for random graphs. How
large must degrees be for the spread to be close to $1/4$?
We shall see that for random regular graphs, what is needed is simply for the
degree $d$
to be big enough.

Firstly we note the deterministic result that the average degree
must be large
in order for the spread to be close to $1/4$, 
and then we give a matching result that for random regular graphs 
graphs high degree 
 is sufficient.
\begin{proposition} \label{lem.avdeg}
  For any fixed $d \geq 2$ there exists $\delta>0$ such that if the
  connected graph $G$ has average degree at most $d$ and $v(G) \geq 3d$ then $\spr(G) \geq 1/4 + \delta$.
  (We can take $\gd=1/(6d)$.) 
\end{proposition}

\begin{proof} 
We shall show that if $V(G)=[n]$ then
\begin{equation} \label{eqn.maxdeglb}
\spr(G) \geq
\frac14 + \Bigpar{\frac{1}{d}-\frac{2}{n}}\Bigpar{1-\frac{1}{d}}.
\end{equation}
Note that this gives $\spr(G) \geq 1/4 +  \frac{1}{6d}$
if $d \geq 2$ and $n \geq 3d$.

Let $t=\lfloor \frac{n}{2d}\rfloor$, let $T$ consist of $t$
vertices of least degree, and let $U$ be the set of vertices
adjacent to a vertex in $T$.  Note that $|U| \leq n/2$. Let $A
\subseteq [n] \setminus T$ be such that $A \supseteq U \setminus
T$ and $|A|= a\=\lfloor \frac{n}{2}\rfloor$.  Let $B=[n]\setminus
(T \cup A)$.

Let $f(v)=0$ on $B$, $1$ on $A$ and $2$ on $T$.
For $X$ uniformly distributed over the vertices, and writing $f$ for $f(X)$,
we have 
$\E f = (1/n) (a+2t)$  and $\E f^2 = (1/n) (a+4t)$,
and hence 
\begin{align*}
\Var(f) & =  (1/n)(a+4t) - (1/n^2)(a^2+4at +4t^2)\\
& =
\frac{a}{n}(1-\frac{a}{n}) + \frac{4t}{n} - \frac{2t}{n} +
\frac{2t}{n^2} \ett{n \text{ odd}} - \frac{4t^2}{n^2}\\
& =
1/4 -\frac{1}{4n^2} \ett{n \text{ odd}} + \frac{2t}{n} +
\frac{2t}{n^2} \ett{n \text{  odd}} - \frac{4t^2}{n^2}\\
& \geq
1/4 + \frac{2t}{n}(1-\frac{2t}{n})\\
& \geq
1/4 + \Bigpar{\frac{1}{d}-\frac{2}{n}}\Bigpar{1-\frac{1}{d}}.
\qedhere
\end{align*}
\end{proof}

  To prove Theorem~\ref{thm.reghighdeg} we need an expansion result for 
  random regular graphs with high degree. Given $\beta>1$ and
$0<\eta<1$ let us say that a graph $G=(V,E)$ has
$(\beta,\eta)$-{\em expansion} if for each $T \subset V$ with $|T|
\leq (1-\eta) |V|/ \beta$ we have $|T \cup N(T)| \geq \beta |T|$.

\begin{lemma} \label{lem.regbigexp}
For each $\beta>1$ and $0<\eta<1/2$ there exists $d_0$ such that for all $d
\geq d_0$ \whp\ 
\gnd\ has $(\beta,\eta)$-expansion.
\end{lemma}

\begin{proof}
We consider the configuration model for \gnd. Let $\alpha>0$
($\alpha$ large). For a positive integer $t$ let $f_{n,d}(t)$ be
the expected number of pairs $T$ and $U$ of sets of disjoint cells
where $|T|=t$ and $|U|= u := \lfloor \alpha t \rfloor$, and each
neighbour of a stub in a cell in $T$ is in $T \cup U$. Let $t_0 =
\lfloor (1-\eta) |V|/ \beta \rfloor$.  We aim to upper bound this
quantity $f_{n,d}(t)$, in order to show that $\sum_{t=1}^{t_0}
f_{n,d}(t) = o(1)$.  The lemma will then follow, with $\beta = 1 +
\alpha$.

Note first that,
since $\frac{d(t+u)-j}{dn-j} \leq \frac{t+u}{n}$ for each $0<j<dn$,
the probability that
each neighbour of a stub in a cell in $T$ is in $T \cup U$ is at most
$(\frac{t+u}{n})^{dt/2}$.
(If we choose the neighbours of the $dt$ stubs in cells in $T$ first,
we have to make at least $dt/2$ such choices.) 
Hence
\begin{align*}
f_{n,d}(t) & \leq
\binom n t \binom{n} u \left(\frac{t+u}{n}\right)^{dt/2}\\
& \leq
\left(\frac{ne}{t}\right)^t \left(\frac{ne}{u}\right)^u \left(\frac{t+u}{n}\right)^{dt/2}\\
& \leq
\left(\frac{ne}{t}\right)^t \left(\frac{ne}{\alpha t}\right)^{\alpha t} \left(\frac{(1+\alpha) t}{n}\right)^{dt/2}\\
& =
\left( e^{1+\alpha} \alpha^{-\alpha} (1+\alpha)^{d/2} t^{d/2 - 1-\alpha} n^{1+\alpha -d/2} \right)^{t}\\
& =
\left( e^{1+\alpha} \alpha^{-\alpha} (1+\alpha)^{1+\alpha}\left(\frac{(1+\alpha)t}{n}\right)^{d/2 - 1-\alpha} \right)^{t}.
\end{align*}
Now $\alpha^{-\alpha} (1+\alpha)^{1+\alpha} =
(1+\alpha)(1+1/\alpha)^{\alpha}
\leq  (1+\alpha) e$.
So 
$$
f_{n,d}(t) \leq \left( (1+\alpha) e^{2+\alpha} \left( \frac{(1+\alpha) t}{n} \right)^{d/2 - 1-\alpha} \right)^{t}.
$$

Let
$\alpha>0$ be sufficiently large that $\log(1+\alpha)+2+\alpha \leq 2\alpha$.
Let $d_0 \geq 6(1+\alpha)$, so that
$d/2 -1 - \alpha \geq d/3$ when $d\ge d_0$. 
For such $d$
$$
f_{n,d}(t)  \leq
\left( e^{2\alpha} \left(\frac{(1+\alpha)t}{n}\right)^{d/3}\right)^t.
$$
If $1 \leq t \leq \log^2 n$, say, then 
$$
f_{n,d}(t) \leq
\left(e^{2\alpha} \left(\frac{(1+\alpha) \log^2 n}{n} \right)^{d/3}\right)^{t}
= O(1/n),$$ since $d\geq 6$.
Also, since
$\frac{(1+\alpha)t}{n} \leq 1-\eta \leq e^{-\eta}$,
for $1 \leq t \leq t_0$ we have
$$
f_{n,d}(t) \leq
\left(e^{2\alpha} e^{-\eta d/3}\right)^{t}.
$$
From these bounds it is easy to complete the proof, with $\beta=1+\alpha$.
\end{proof}
\bigskip

\begin{lemma} \label{lem.Ss}
Let $\beta\ge3$,
 $\eta = \beta^{-1}$ and $n\ge6\gb+\gb^2/2$, 
and let $G=(V,E)$ have $(\beta,\eta)$-expansion.
Let $f$ be an integer-valued function on $V$ with median $0$.
Let $V_{\geq i}$ denote $\{v \in V: f(v) \geq i\}$ and so on.
Assume that $|V_{\geq 1}| \geq |V_{\leq -1}|$.  Then
\begin{equation} \label{eqn.onetwobig}
|V_{\geq i}| \leq \beta^{-(i-1)}n/2 \mbox{ and }
|V_{\leq -i}| \leq 2\beta^{-i} n \mbox{ for each } i \geq 1.
\end{equation}
\end{lemma}

\begin{proof}
Note that $|V_{\leq 0}| \geq n/2$ and $|V_{\geq 0}| \geq n/2$.
Observe also that $N(V_{\geq i}) \subseteq V_{\geq i-1}$.
If $|V_{\geq 2}| > (1-\eta) n/\beta$ then,
choosing a set $T\subset V_{\ge2}$ with
$|T|=\floor{ (1-\eta) n/\beta}$, 
\begin{equation}\label{klm2}
  |V_{\geq 1}| \geq |T\cup N(T)|\ge\gb|T|>(1-\eta)n-\gb
\ge n/2 \geq |V_{\geq 1}|,
\end{equation}
a contradiction:
thus $|V_{\geq 2}| \leq (1-\eta)n/\beta$.  Hence
$|V_{\geq 2}| \leq \frac{1}{\beta} |V_{\geq 1}| \leq  \frac{n}{2 \beta}$, and further for all $i \geq 1$ we have
$|V_{\geq i}| \leq \beta^{-(i-1)} |V_{\geq 1}|$.
Similarly, for all $i \geq 1$ we have
$|V_{\leq -i}| \leq \beta^{-(i-1)} |V_{\leq -1}|$.
Hence it suffices to show \eqref{eqn.onetwobig} for $i=1$, 
i.e., that
$|V_{\geq 1}| \leq n/2$, which is trivial, and
$|V_{\leq -1}| \leq 2n/\gb$.

Recall that $|V_{\geq 1}| \geq |V_{\leq -1}|$.
We consider two cases, depending on the size of $V_{\geq 1}$.
If $|V_{\geq 1}| \leq (1-\eta) n/\beta$ then
$|V_{\geq -1}| \leq |V_{\geq 1}| <2 n/\beta$.
If $|V_{\geq 1}| > (1-\eta) n/\beta$ then
$|V_{\geq 0}| \geq (1-\eta)n-\gb$ as in \eqref{klm2},
so
$|V_{\leq -1}| \leq \eta n+\gb\le 2n/\gb$.
\end{proof}

\noindent The last lemma easily yields:
\begin{lemma} \label{lem.betaspread}
For any $\eps>0$ there exists $\beta>1$ such that each graph $G$
with $(\beta,\beta^{-1})$-expansion
and $n$ large enough 
satisfies $\spr(G) < 1/4 +\eps$.
\end{lemma}

\begin{proof}
Let $f$ be an integer-valued Lipschitz function on $G$. We may assume
that the median of $f$ is 0, and (by symmetry) that
$|V_{\geq 1}| \geq |V_{\leq -1}|$.
Then \refL{lem.Ss} yields, if $\gb\ge3$ and $n$ is large,
\begin{equation*}
\Var(f)
\le
\E\Bigabs{f-\frac12}^2
\le\frac14+\sum_{i\neq0,1}\frac{|V_i|}n(i-1/2)^2
\le\frac14+O(\gb\qw).
\qedhere
\end{equation*}
\end{proof}

Lemmas~\ref{lem.regbigexp} and~\ref{lem.betaspread} complete the
proof of Theorem~\ref{thm.reghighdeg}.


\section{Small worlds}  \label{sec.sw} 

  In this section we consider the small worlds model $R_{n,c/n}$ for $c>0$, and prove Theorem~\ref{thm.sw1}.
  We need some preliminary work so that we can appeal to Lemma~\ref{lem.decexp}.
  The first step is to show that we may assume that $c \geq 2$, by contracting sections of the ring.
  Now, if we delete the edges of the ring randomly, 
keeping each with probability $c/n$, we obtain a random graph $\Gnc$, whose 
giant component $H$ is a well-behaved weak $\alpha$-decorated expander by
\refL{LT2+}.  
  We show that using the ring to join the other vertices to $H$ yields further decorations but \whp\ we still have a
  well-behaved weak
  $\alpha'$-decorated expander. 

\step1{Reduction to the case $c=2$.}
  We start with a deterministic lemma,
  which will show that the spread does not shrink too much when we contract sections of the ring.

\begin{lemma} \label{lem.contract}
   Let $G$ be a connected graph on $V$ where $|V|=n$, let $k$ be an integer with $1 \leq k <n$, and let
   $\bar{n}=\lfloor n/k \rfloor$.
   Let $V_1,\ldots,V_{\bar{n}}$ be a partition of $V$ such that each induced subgraph $G[V_i]$ is a connected graph
   with $k$ or $k+1$ vertices.
   Form the graph $\bar{G}$ on $[\bar{n}]$ by contracting each $V_i$ to a
   single new vertex $i$. Then 
$$\spr(G) \leq \frac{(k\!+\!1)^3}{k} \spr(\bar{G}) + \frac{k^2}{4}.$$
\end{lemma}

\begin{proof}
  Let $f$ be a Lipschitz function for $G$, with mean $\mu$.  Let $\mu_i$ be the mean of $f|_{V_i}$, that is
  $\mu_i = (1/|V_i|) \sum_{w \in V_i} f(w)$. Then, by a standard decomposition of variance,
\begin{equation}  \label{eqn.twoterms}	
  \begin{split}
 \Var(f) &=
   \frac1{n} \sum_i \sum_{w \in V_i} (f(w)-\mu_i+\mu_i-\mu)^2 \\
 &=
  \frac1{n} \sum_i \left( \sum_{w \in V_i} ((f(w)-\mu_i)^2 + (\mu_i-\mu)^2) \right)  \\
 &=
  \sum_i  \frac{|V_i|}{n}\Var(f|_{V_i}) + \sum_i \frac{|V_i|}{n}
  (\mu_i-\mu)^2.
  \end{split}
\end{equation}
  We consider the two terms here separately.
  Since the induced subgraph $G[V_i]$ has diameter at most $k$, 
  $\Var(f|_{V_i}) \leq \spr(G[V_i]) \leq k^2/4$. Thus
$$ \sum_i  \frac{|V_i|}{n}\Var(f|_{V_i}) \leq \frac{k^2}{4}.
$$
  
  Now consider the second term above.
  Observe that for each $i$ and each $w \in V_i$, $|f(w)-\mu_i| \leq k/2$.
  Thus if $i$ and $j$ are adjacent in $\bar{G}$ then $|\mu_i-\mu_j| \leq k+1$.
  Let $\bar{f}(i)=\mu_i$ for each $i \in [\bar{n}]$.
  Then $(1/(k+1)) \bar{f}$ is Lipschitz for $\bar{G}$, and so
$$  \Var(\bar{f}) \leq (k+1)^2 \spr(\bar{G}). $$
  Next, let $\nux=(1/\bar{n}) \sum_i \mu_i$; 
let $h(i)= \mu_i-\nux$ for each $i \in [\bar{n}]$;
  and let the random variable $X$ take values in $[\bar{n}]$ with $\P(X=i)=|V_i|/n$.  Observe that $\E[\bar{f}(X)]= \mu$.
  Then
\begin{align*}
  \sum_i \frac{|V_i|}{n} (\mu_i-\mu)^2 & = 
  \Var \bar{f}(X) \; =  \; \Var h(X) \; \leq \; \E[h(X)^2] \\
  & \leq  \sum_i \frac{k\!+\!1}{n} (\mu_i-\nux)^2 \; \leq \; \frac{k\!+\!1}{k} \Var(h)\\
  & =  \frac{k\!+\!1}{k} \Var(\bar{f}) \; \leq \; \frac{(k\!+\!1)^3}{k} \spr(\bar{G}).
\end{align*}
  Now~(\ref{eqn.twoterms}) and the above bounds let us complete the proof.
\end{proof}

  With the above lemma in hand, we may quickly complete the reduction to the case $c=2$.
  Consider $0<c<2$.  Fix a positive integer $k$ with $k^2 c > 2(k+1)$.  
Observe that given two disjoint $k$-subsets of $[n]$,
  the probability that there is an edge in $G_{n,c/n}$ between the sets is $1-(1-c/n)^{k^2} = k^2c/n + O(1/n^2)$.
  Assume that \whp\ $\spr(R_{n,2/n}) \leq b$ for some constant $b$.
  
  Consider a large $n$, partition the vertex set of $\Cn$ into paths of $k$ or $k+1$ vertices
  (which we can always do once $n \geq k(k-1)$),
  and from $R_{n,c/n}$ form the corresponding contracted graph as in Lemma~\ref{lem.contract}.
  Call the contracted graph  $R_{\bar{n}}$.
  Then $R_{\bar{n}}$ contains a deterministic  Hamilton cycle arising from the cycle $\Cn$;
  and edges not in the cycle appear independently, each with probability at
  least $2(k+1)/n\ge2/\bar n$.
  Thus \whp\ $\spr(R_{\bar{n}}) \leq b$ by the assumption above.
  Hence by Lemma~\ref{lem.contract} \whp{} $\spr(R_{n,c/n}) \leq \frac{(k\!+\!1)^3}{k} b + \frac{k^2}{4}$.
  This completes the reduction to the case $c=2$. 

\step2{Joining the other vertices to the giant component $H$ of $G_{n,2/n}$.}
Let us think of $R_{n,c/n}$ as generated by starting with $G_{n,c/n}$ on
vertex set $[n]$, 
  picking an independent uniform random Hamilton cycle $C$ in the complete
  graph on $[n]$, 
  and adding the edges of $C$ if they are not already present.
We shall see that adding some edges of $C$ to the edges of $H$ \whp{} yields a
  well-behaved weak 
  $\alpha$-decorated expander $G^+$ on $[n]$ 
  (for a suitable fixed value of $\alpha>0$).  
  
  Condition on $H$ being a fixed well-behaved $\alpha_0$-decorated expander 
  (for some fixed $\alpha_0>0$), fix a corresponding subgraph $F$, and
  let $D_1,\ldots,D_{\ellx}$ be the decorations.
As usual, let $n':=v(H)$. 
We further assume (using \refL{Lgnc} and \refR{RP3}) that $H$ satisfies (P1)
and that  (P3) holds in the stronger version
$v_i(H) \leq e^{-2i} v(H)$ for all $i\ge\CCgnc$.

Discard all edges outside $H$ other than those from the random cycle $C$,
which we take as oriented. 
  The vertices in $V(H)$ divide the remaining vertices into $n'$ paths:
  for each vertex $w$ in $V(H)$ let $Q_w$ be the maximal path of vertices outside $H$ ending at $w$ 
  with $X_w \geq 0$ vertices (not counting $w$). 
We attach the path $Q_w$ at $w$ for each $w$ in $H$, forming the graph $G^+$. 
If $w$ is in $V(F)$ then we have one new decoration attached at $w$ (if
$X_w>0$). 
  If $w$ is in decoration $D_i$ then we add $X_w$ vertices and edges to
  $D_i$ (and no extra edge to $E(D_i,F)$).

  The properties (DE1), (DE3), (DE4), (P2), (P3) are easily seen to hold for $G^+$
 (for a suitable value of $\alpha>0$). We must check that also \DEij\  holds.

  What is the distribution of $(X_w:w \in V(H))$?
  We may assume without loss of generality that vertex $n$ is in $V(H)$.
  Think of the vertices in $[n]$ as white.  Re-colour vertex $n$ black.
  Choose a uniformly random subset $S$ of $[n-1]$ of size $n'-1$, and re-colour these elements black.
  Let $\tilde{X}_1$ be the number of white elements before the first black one; 
  and for $i=2,\ldots,n'$ let $\tilde{X}_i$ be the number of white elements
  between the $(i-1)$th black vertex and  the $i$th. 
  The distribution of $(X_w:w \in V(H))$ is the same as that of $(\tilde{X}_1,\ldots,\tilde{X}_{n'})$.
  Thus for each list $k_1,\ldots,k_{n'}$ of non-negative integers
  with $\sum_i k_i=n-n'$ we have
$$ 
\P(X_i=k_i \mbox{ for each } i) =  
\binom{n-1}{n'-1}\qw.
$$ 
  It follows (see for example~\cite{dr98})
  that the family $(X_w:w \in V(H))$ is negatively associated. 
  Also $X_w \leq_s \tilde{X}$ for each $w$, where $\tilde{X}$ is geometric
  with parameter $p' = \frac{n'-1}{n-1}$ (and mean $1/p'-1$),  
  and $\le_s$ means stochastic ordering (\ie, there exists a coupling such that $X_w\le X$). 
  But $p' \geq p:=\gamma/3$ for $n$ sufficently large by (P1). 
  Then $X_w \leq_s X$ where $X$ is geometric with parameter $p$ (note that this value $p$ is fixed). 

  Let $\cA = (A_i: i \in I )$ be the partition of $V(H)$ into the vertex sets $V(D)$ of the decorations $D$ of $H$
  together with the singletons $\{w\}$ for each $w \in V(F)$.
  Thus $|I| = \ellx + v(F)$. 
For each $i$, let $D_i^+:=\bigcup\set{Q_w:w\in A_i}$ be the (possibly empty)
union of the  paths $Q_w$ attached to $A_i$.

  Let $Y_i := v(D^+_i)=\sum_{w \in A_i} X_w$ for each $i \in I$.
  The family $(Y_i:i \in I)$ is negatively associated, since it is formed by taking sums of disjoint
  members of $(X_w:w \in V(H))$, see~\cite{jdp83}. 
  Hence, letting $M_X(t):=\E e^{tX}$ denote the moment generating function,
 $M_{Y_i}(t) \leq M_X(t)^{|A_i|}$. 
  Also the family $(\ett{Y_i \geq j} : i \in I)$ is negatively associated
  for each $j$. 
  Let 
  $$I_j := \{i \in I: |A_i| \leq (p/2)j\},$$ and
  $$Z_j:= \sum_{i \in I} \ett{Y_i \geq j} \;\; \mbox{ and } \;\; Z'_j:= \sum_{i \in I_j} \ett{Y_i \geq j}.$$
  Note that for each $j \geq 2/p$
\begin{equation} \label{eqn.ring1}
  Z_j - Z'_j \leq |\{i \in I : |A_i| > (p/2)j\}|  
  \leq e(H)\, e^{-\alpha_0 (p/2)j}
\end{equation}
  by \DEij \, for $H$.

  Observe that if $f(t)= e^{-t} M_X(t)^{p/2}$ then 
$$ 
\frac{d}{dt} ( \log f(t))\bigm|_{t=0} 
=-1+ \frac{p}{2} \frac{M'_X(0)}{M_X(0)}
=-1+ \frac{p}{2} \E X
 = - \frac{1+p}2<0 ,
$$
  and so there exists $t_0>0$ such that $f(t_0)<1$.
   Let $\alpha = - \log f(t_0)$ so $\alpha>0$ and $f(t_0)=e^{-\alpha}$.
  For each $i \in I_j$, by Markov's inequality
  \begin{equation}\label{eyij}
	\begin{split}
\E[\ett{Y_i \geq j}] &=
   \P(Y_i \geq j) \; \leq \; e^{-t_0 j} M_{Y_i}(t_0)\\
   & \leq  e^{-t_0 j} M_{X}(t_0)^{|A_i|} 
\; \leq \; (e^{-t_0} M_{X}(t_0)^{p/2})^j \; = \; e^{-\alpha j}.	  
	\end{split}
  \end{equation}
  Since the family $(\ett{Y_i \geq j} : i \in I)$ is negatively associated,
\begin{equation*}
  M_{Z_j'}(t) \le \prod_{i\in I_j} M_{ \ett{Y_i \geq j}}(t)
  \le M_{\Be(e^{-\ga j})}^{|I_j|}(t) =M_{\Bi(|I_j|,e^{-\alpha j})}(t)
\end{equation*}
  for each $t \geq 0$. Hence, the usual Chernoff estimates for the upper tail
  for $\Bi(|I_j|,e^{-\alpha j})$ apply to $Z_j'$ too, and thus,
  see for example Corollary 2.4 (and its proof, \cf{} Theorems 2.8 and 2.10)
  of~\cite{JLR},
\begin{equation} 
  \P\bigpar{Z'_j \geq 2 |I| e^{-(\alpha/3)j}}
   \leq \exp \bigpar{ - \tfrac13 |I| e^{-(\alpha/3)j}}.
\end{equation}
  For $j \leq (2/\alpha) \log |I|$, we have
  $|I| e^{-(\alpha/3)j}\ge |I|\qqq \ge (\cc n)\qqq$, and thus
\begin{equation} \label{eqn.ring3a}
  \P\bigpar{Z'_j \geq 2 |I| e^{-(\alpha/3)j}}
  \leq \exp \bigpar{ - \cc n\qqq}.\ccdef\ccring
\end{equation}
  For $j > (2/\alpha) \log |I|$, we use Markov's inequality and \eqref{eyij},
  which yield
\begin{equation}  \label{eqn.ring3b}
  \P(Z_j'>0)\le\E Z'_j \leq |I| e^{-\alpha j}. 
\end{equation}
  Summing \eqref{eqn.ring3a} or \eqref{eqn.ring3b} for $j\ge0$ yields
\begin{equation*} 
  \begin{split}
\P \Bigpar{ Z'_j > 2 |I| e^{-(\alpha/3)j} & \mbox{ for some } j\ge0 }
\\&
\le \sum_{j\le (2/\ga)\log |I|} \exp \bigpar{ - \ccring n\qqq}
+\sum_{j> (2/\ga)\log |I|} |I| e^{-\alpha j}
\\&
=o(1)+O(|I|\qw)	\to0
  \end{split}
\end{equation*}
as \ntoo, since $|I|\ge \cc n$.
  Hence 
$$ \P(Z'_j < 2 |I| e^{-(\alpha/3)j} \mbox{ for each } j=0,1,\ldots) \to 1 \; \mbox{ as } n \to \infty. $$
  This result together with~(\ref{eqn.ring1}) shows that, for some fixed $\alpha'>0$, 
$$ \P(Z_j < 3e(H) e^{-\alpha' j} \mbox{ for each } j=0,1,\ldots) \to 1 \;
  \mbox{ as } n \to \infty, $$
which by \DEij{} for $H$ easily implies that \DEij{} holds for $G^+$ \whpx.

  Hence, \whp{} $G^+$ is a well-behaved weak $\ga$-decorated expander 
and we may use Lemma~\ref{lem.decexp} to see that then \eqref{t3} holds 
for $G^+$, and consequently for $\Rnc$, which
completes the proof of Theorem~\ref{thm.sw1}. 
\smallskip

In Lemma~\ref{LT2}, we may insist that the giant component satisfies (DE2) rather than just \DEij, see~\cite{BKW}.
  Using this result, it is not hard to adapt the above proof to deduce that \whp{} $G^+$ satisfies (DE2) rather than just \DEij.
  \smallskip
  
  To set Theorem~\ref{thm.sw1} in context, note that for any $K$ 
  there exists a constant $c_0>0$
  such that if $0< c \leq c_0$ then \whp\ $\; \spr(R_{n,c/n}) >K$.  Indeed we have the following result.

\begin{proposition} \label{prop.sw}
  For any $K$ there exists a constant $\eps>0$ such that the following holds.
  Let $G_n$ be formed from the cycle $\Cn$ by adding at most $\eps n$ edges.  Then $\spr(G_n) \geq K$ for
  $n$ sufficiently large.  
\end{proposition}

\begin{proof}
  Let $t:=\ceil{\sqrt{6K}}$,
and assume that $0<\eps \leq \frac1{16 t}$.  We shall show that $\spr(G_n)
\geq t^2/6\ge K$
  if $n$ is sufficiently large.
  
  We first define a Lipschitz function for $\Cn$. 
  It is convenient to let the vertex set of $\Cn$ be $V=\{0,1,\ldots,n-1\}$. Divide $V$ into $\lfloor n/4t \rfloor$
  sections $\{0,\ldots,4t-1\}$, $\{4t,\ldots,8t-1\}$, \dots{} 
plus a `remainder' (possibly).
  If $i \in V$ satisfies $i \equiv j \pmod{4t}$ where $0 \leq j <4t$, we set
  $f(i)=j$ if $0\leq j <t$, $f(i)=2t-j$ if $t \leq j<3t$ and $f(i)=j-4t$ if $3t \leq j<4t$.
  (Thus on the section $\{0,\ldots,4t-1\}$, $f$ increases from 0 to $t$, then decreases from $t$ to $-t$
  and then increases to $-1$, always taking unit steps.)  Observe that
  $$\sum_{j=0}^{4t-1} f(j)^2 = 4 \sum_{j=0}^{t-1} j^2 + 2 t^2 = \frac43 t(t^2+\frac12).$$
  
  Now re-set $f(v)$ to 0 for each $v$ in the `remainder' (that is, for $4\lfloor n/4t\rfloor \leq v \leq n$),
  and for each $v$ in a section which contains a vertex of degree $>2$ in~$G_n$.  Then $f$ is a Lipschitz function for $G_n$,
  and $f$ is unchanged on at least $\frac{n}{4t} -1 - 2 \eps n \geq \frac{n}{8t}-1$ sections.
  Now $\sum_{v \in V}f(v)=0$; and 
$$\sum_{v \in V} f(v)^2 \geq (\frac{n}{8t}-1) \frac43 t(t^2+\frac12) \geq n \cdot t^2/6$$
  for $n$ sufficiently large.
  Then $\Var(f) \geq t^2/6$, and so $\spr(G_n) \geq t^2/6\ge K$, as
  required. 
\end{proof}


\section{$\Kn$ with random edge-lengths}
\label{sec.lengths}
 
  In this section we prove Theorem~\ref{thm.lengths}.
  Given $\alpha>0$, following~\cite{BKW}, we say that a family $\cA=(a_i:i \in I)$ of non-negative numbers
  has an $\alpha$-\emph{exponential tail} if 
  $$ \frac{|\{i \in I: a_i \geq j\}|}{|I|} \leq 2 e^{-\alpha j} \;\; \mbox{  for all } j \geq 0.$$ 
  We need two lemmas.
  
\begin{lemma} \label{lem.biptail} 
  For each $\lambda>0$ there is an $\alpha>0$ such that the following holds.
  Consider $K_{n,n}$ with independent edge-lengths $X_e$, where $X_e$ is exponentially distributed
  with parameter $\lambda/n$ (and thus mean $n/\lambda$).
  Then \whp{} there is a perfect matching such that the edge-lengths have an $\alpha$-exponential tail.
\end{lemma}

\begin{proof}
  This follows from the result of Walkup~\cite{walkup80} that,
  if each vertex independently and uniformly picks arcs to two vertices in the other part,
  then \whp{} there is a perfect matching using only such arcs (ignoring orientations).
  (With minimal changes, we could allow each vertex to pick 3 arcs instead of 2, and then the corresponding weakened version of
  Walkup's result follows directly from Hall's Theorem, see~\cite{walkup80}.)
  
  Replace each edge of $K_{n,n}$ by a pair of oppositely directed arcs.
  Let the arcs $e$ have independent edge-lengths $X'_e$, each exponentially distributed with parameter $\frac\lambda{2n}$.
  For each edge $e$ of $K_{n,n}$, we may assume that $X_e$ is the minimum of $X'_{e1}$ and $X'_{e2}$ where $e1$ and $e2$ are
  the two arcs arising from orienting $e$.
  Let $S$ be the set of $4n$ arcs formed from the 2 shortest arcs leaving each vertex.
  
By Walkup's result \cite{walkup80}, 
there is \whp{} a perfect matching using only arcs in $S$.
Let $Z_j= |\{e \in S: X'_e \geq j \}|$. 
  It will suffice to show (by changing $\ga$) that
\begin{equation} \label{eqn.Zjshow}
  \P\left( Z_j < 16n e^{-\alpha j/3} \mbox{ for each } j \right) \to 1
\qquad  \mbox{as } n \to \infty.
\end{equation}

  Let $Y_n$ be the second smallest of $n$ independent random variables $\tilde{X}_1,\ldots,\tilde{X}_n$ which are each
  exponentially distributed with parameter $\frac{\lambda}{2n}$. Let $p= \P(\tilde{X}_1 < j) = 1-e^{-\frac{\lambda j}{2n}}$,
  and observe that $p \leq \frac{\lambda j}{2n}$. Then
\begin{align*}
  \P(Y_n \geq j) = \P(\Bin(n,p) \leq 1) & = (1-p)^n + n p (1-p)^{n-1}\\
&\le
(1+ \frac12 \lambda j e^{\frac{\lambda j}{2n}}) \cdot e^{-\frac{\lambda j}{2}}.
\end{align*} 
  Thus there is a constant $\alpha>0$ such that
  \begin{equation}
	\label{bua}
\P(Y_n \geq j) \leq 2 e^{-\alpha j} \qquad \mbox{for each } j \geq 0. 
 \end{equation}

  Let $\tilde{Y}_1,\ldots,\tilde{Y}_{2n}$ be independent, each distributed like $Y_n$.
    Let $\tilde{Z}_j:= \sum_{i=1}^{2n} \ett{\tilde{Y}_i \geq j}$; and note that
$Z_j \leq_s 2 \tilde{Z}_j$ (recall that $\leq_s$ denotes stochastic
	domination).   
Then, by \eqref{bua},
$$ \tilde{Z}_j \leq_s \Bin(2n, 1 \wedge 2e^{-\alpha j}).
$$ 
  The remainder of the proof is quite similar to the end of the proof of Theorem~\ref{thm.sw1}.
  By a Chernoff estimate, 
$$ \P(\tilde{Z}_j \geq 8n e^{-\alpha j/3}) 
\leq \exp \bigpar{-\tfrac16\cdot{8n}e^{-\alpha j/3}}.
$$
  When $j\le(2/\ga)\log n$ this is $\le \exp({-n\qqq})$.
For larger $j$ we simply use Markov's inequality:
$$
\P(\tZ_j>0) \leq \E \tZ_j \leq 4n e^{-\alpha j}
$$
and thus
$$
\sum_{j \ge (2/\alpha) \log n}
\P(\tZ_j>0) \leq \frac 4{1-e^{-\ga}}n\qw.
$$
It follows that
$$ \P\left( \tilde{Z}_j \ge 8n e^{-\alpha j/3} \mbox{ for some } j\ge0  \right)
\le
\sum_{j=0}^\infty\P(\tilde{Z}_j \geq 8n e^{-\alpha j/3}) 
 \to 0 
$$
as \ntoo,  and~(\ref{eqn.Zjshow}) follows, completing the proof.
\end{proof}

\begin{lemma} \label{lem.biptail2}
  Fix $0<\gamma<1$.  Fix $\lambda>0$.  Then there is an $\alpha>0$ such that the following holds.
  Consider a complete bipartite graph $K_{a,b}$, with parts $A$ of size $a$ and $B$ of size $b$, where
  $\gamma n \leq a, b \leq n$.  Let the edges $e$ have independent lengths $\ell(e)$, each exponentially
  distributed with parameter $\lambda/n$.
 Then \whp{} there is a set of edges $S = \{ u w\}\subset A\times B$ such
 that 
  \begin{itemize}
  \item [(a)] $|\{ u \in A: uw \in S \}| =1$ for each $w\in B$,
  \item [(b)] $|\{ w \in B: uw \in S \}| \leq \lceil b/a \rceil$ for each $u
	\in A$,
\item [(c)] the family $(\ell(u w): uw \in S)$ has $\alpha$-exponential tails.	
  \end{itemize}
\end{lemma}

\begin{proof}
  By considering adding at most $a$ vertices to $B$, we see that it suffices to consider the case $a|b$.
  But now we see that it suffices to assume that $a=b$, and so the result follows from the last lemma.
\end{proof}

  Now we are ready to prove Theorem~\ref{thm.lengths}.
  If $X$ is uniform on $[n]$ and $Y$ is exponentially distributed with parameter $\frac{1}{n}$,
  then $X \leq_s 1+Y$; thus we may assume that edge-lengths are \iid{}, each distributed like $1+Y$.
  Next, replace each edge $e$ by a blue copy $e_B$ and a red copy $e_R$,
  and give these copies \iid{} edge-lengths, each distributed like $1 + Y'$
  where $Y'$ is 
  exponentially distributed with parameter $\frac1{2n}$.
  We may assume that the length $\ell(e)$ of $e$ is the smaller of the lengths of $e_R$ and $e_B$.
  Note that if $b=b(n)= - 2n \log (1-\frac2{n})$ then $b \sim 4$ and
$$ \P(Y' \leq b) = 1 - e^{-\frac{b}{2n}} = \frac{2}{n}.$$
  Thus by keeping only blue edges with an appropriate length $b+1 \leq 6$ (for large $n$)
  we may generate a random graph $G_{n,2/n}$.
  
For some $\ga_1>0$, \whp{} there is in this random graph
a giant component $H$ and a subgraph $F$ showing that $H$ is a
  well-behaved weak $\alpha_1$-decorated expander, see \refL{LT2+}. 
Condition on there being such an $H$ and $F$, and fix them. 
We also assume that (P1) holds, \ie, $n'>\gam n/2$, see \refL{Lgnc}.
Thus, 
$v(F) \geq \alpha_1 n'\ge\cc n$.
  List the decorations as $D_1,\ldots,D_{\ellx}$.  Let $W = [n] \setminus V(H)$.

  Now we use the red edges.
  By Lemma~\ref{lem.biptail2} applied to the red edges between $V(F)$ and $W$, 
  there is a set $S$ of red edges $\{ uw\}\subset V(F)\times W$
  such that 
  \begin{itemize}
  \item [(a)] $|\{ u \in V(F): uw \in S \}| =1$ for each $w\in W$,
  \item 
[(b)] $|\{w \in W: uw \in S\}| \leq \lceil \xfrac{|W|}{|F|}\rceil 
\leq \xfrac{1}{\alpha}$   for each $u \in V(F)$,
\item [(c)] the family $(\ell(uw): uw \in S)$ has $ \alpha_2 $-exponential
  tails (for a suitable $\alpha_2>0$). 
  \end{itemize}

  Let $G$ be the graph on $[n]$ with edge set $E(H) \cup S$.  We still have
  the subgraph $F$ and decorations 
  $D_1,\ldots,D_{\ellx}\; $; but now for each edge $uw \in S$ we have a new
  one-vertex decoration $\{w\}$ decorating $u \in V(F)$. 
  For $i=1,\ldots,\ellx$ let $\tilde{v}(D_i) = v(D_i)$; and for each $w \in
  W$ let $\tilde{v}(\{w\}) = \ell(uw) \ (\geq 1)$, 
  where $uw \in S$.
  Now use $D_i$ to refer to any of the $\ellx+ |W|$ decorations of $G$.
Then $G$ is a well-behaved weak $\alpha_3$-decorated expander, for a suitable
  $\ga_3$, 
  except that in condition \DEij, \ $v(D_i)$ is replaced by $\tilde{v}(D_i)$.
  
To show that each Lipschitz function for $G$ then satisfies
inequality~(\ref{t3}), 
which implies \refT{thm.lengths}, we follow the proof of Lemma~\ref{lem.decexp}.
  We need no changes until just after inequality~(\ref{t2b}) when the proof starts to deal with decorations.
  From there until inequality~(\ref{3i4}) replace each appearance of $v$ by $\tilde{v}$.  Now the proof works just as before.


\section{Open problems}\label{Sproblems}
  We saw in the preceding sections that high degree is precisely what is needed
  to force the spread of the random regular graph $\gnd$ to be close to $1/4$.
  We believe that a corresponding result should hold for the giant component $\Hnc$ of $\Gnc$.
\begin{problem} \label{problem.highdegH}
  Is it the case that for each $\eps>0$, there exists $c_0$ such that for each $c \geq c_0$
  \whp\ $\spr(\Hnc) < 1/4 + \eps$?
\end{problem}
If \refL{LT2} holds uniformly (in the sense that for any $c > 1$,
$\alpha = \alpha(c)$ can be chosen such that the conclusions of the theorem hold
in $H_{n,c'/n}$ for all $c' \geq c$, with this value of $\alpha$) then the proof of Theorem \ref{T3} can
be modified to yield an affirmative answer to the above question.
This uniformity seems very likely to hold, but does not follow immediately from the proof of  \refL{LT2} given
in \cite{BKW}.
  
There is a natural similar question for the `small world' random graph
$R_{n,c/n}$, 
  to complete the picture described in Theorem~\ref{thm.sw1} and
  Proposition~\ref{prop.sw}. 
  
\begin{problem} \label{problem.highdegR}
  Is it the case that for each $\eps>0$ there exists $c_1$ such that for each $c \geq c_1$ 
  \whp\ $\spr(R_{n,c/n}) < 1/4 + \eps$?
\end{problem}

\refT{C1} suggests that the spread of $\gnd$ might converge (in
probability) to a constant, and similarly for Theorem~\ref{C2} and \Hnc.

\begin{problem}
Do there exist constants $\alpha_d$ for each $d \geq 3$ and $\beta_c$ for each $c>1$ such that $\spr(\gnd)\pto \alpha_d$ and $\spr(\Hnc) \pto \beta_c$ as $\ntoo$?
\end{problem}

\noindent
We know that if the constants $\alpha_d$ exist then they are (weakly)
decreasing in $d$ and tend to $1/4$
as $d \to \infty$. It seems likely that the analogous results should hold for $\Gnc$.

\begin{problem}
  If the constants $\gb_c$ exist, are they decreasing in $c$, and do they tend to $1/4$ as $c \to \infty$?
\end{problem}
\noindent
  Again, there is a natural similar question for $R_{n,c/n}$.
\smallskip

For $R_{n,c/n}$, we can also ask about the constant $\CCsw(c)$ 
in \refT{thm.sw1} as $c\to0$. 
Proposition \ref{prop.sw} shows that $\CCsw(c)\to\infty$ as $c\to0$.
The proof of~\refT{thm.sw1} in \refS{sec.sw} yields, through the argument in Step 1 of the proof
with, for example, $k=\ceil{3/c}$, that we can take $\CCsw(c)=O(c\qww)$ as
$c\to0$. We conjecture that this is best possible, in analogy with
\refT{behaves} for $\Gnc$.

\begin{problem}
  Is the optimal $\CCsw(c)=\Theta(c\qww)$ as $c\to0$?
\end{problem}

  In the small worlds model $R_{n,c/n}$ we start with the deterministic cycle $\Cn$ and add edges independently with
  probability $c/n$.  Suppose that we start instead with a deterministic graph $G_n$ on $[n]$:
  let us denote the corresponding random graph by $R(G_n, \frac{c}{n})$, so $R_{n,c/n}$ is $R(\Cn, \frac{c}{n})$.
  For example a popular small worlds model takes $G_n$ as a power $\Cn^r$ of $\Cn$, 
  where two vertices are adjacent in $\Cn^r$ if they are at distance at most $r$ in $\Cn$.
  
  We may adapt the proof of Theorem~\ref{thm.sw1} to show the same result when $G_n$ is the $n$-vertex path $\Pn$;
  that is, there is a constant $\CC=\CCx(c)>0$ such that \whp{} $\spr(R(\Pn, \frac{c}{n})) \leq \gamma$.
  (Indeed, we could take $G_n$ as $\Cn$ less any set of edges which are at distance at least $\CC \log n$ apart,
  for a sufficiently large constant $\CCx$ depending on $c$.  For we could think of these edges as simply being coloured red,
and \whp{} no two red edges appear on the cycle in the same path $Q_w$  
between vertices in $V(H)$. If there is a red edge in $Q_w$, we of course
join the part before the red edge to $H$ by an edge from the first vertex in
$Q_w$ to its predecessor in the cycle.)
  
It seems likely that starting with the path $\Pn$ is the worst case, which
leads to the following problem.
  
\begin{problem}
  Is it the case that \whp{} $\spr(R(G_n, \frac{c}{n})) \leq \CC(c)$ for
  every sequence $G_n$ 
  of connected graphs on $[n]$ and every $c>0$?
\end{problem}




\newcommand\AAP{\emph{Adv. Appl. Probab.} }
\newcommand\JAP{\emph{J. Appl. Probab.} }
\newcommand\JAMS{\emph{J. \AMS} }
\newcommand\MAMS{\emph{Memoirs \AMS} }
\newcommand\PAMS{\emph{Proc. \AMS} }
\newcommand\TAMS{\emph{Trans. \AMS} }
\newcommand\AnnMS{\emph{Ann. Math. Statist.} }
\newcommand\AnnPr{\emph{Ann. Probab.} }
\newcommand\CPC{\emph{Combin. Probab. Comput.} }
\newcommand\JMAA{\emph{J. Math. Anal. Appl.} }
\newcommand\RSA{\emph{Random Struct. Alg.} }
\newcommand\ZW{\emph{Z. Wahrsch. Verw. Gebiete} }
\newcommand\DMTCS{\jour{Discr. Math. Theor. Comput. Sci.} }

\newcommand\AMS{Amer. Math. Soc.}
\newcommand\Springer{Springer-Verlag}
\newcommand\Wiley{Wiley}

\newcommand\vol{\textbf}
\newcommand\jour{\emph}
\newcommand\book{\emph}
\newcommand\inbook{\emph}
\def\no#1#2,{\unskip#2, no. #1,} 
\newcommand\toappear{\unskip, to appear}

\newcommand\webcite[1]{
\texttt{\def~{{\tiny$\sim$}}#1}\hfill\hfill}
\newcommand\webcitesvante{\webcite{http://www.math.uu.se/~svante/papers/}}
\newcommand\arxiv[1]{\webcite{arXiv:#1.}}

\def\nobibitem#1\par{}


\begin{thebibliography}{21}
\providecommand{\natexlab}[1]{#1}
\providecommand{\url}[1]{\texttt{#1}}
\expandafter\ifx\csname urlstyle\endcsname\relax
  \providecommand{\doi}[1]{doi: #1}\else
  \providecommand{\doi}{doi: \begingroup \urlstyle{rm}\Url}\fi

\bibitem[Alon et~al.(1998)Alon, Boppana, and Spencer]{alon1998aii}
N.~Alon, R.~Boppana, and J.~Spencer.
\newblock An asymptotic isoperimetric inequality.
\newblock \emph{Geom. Funct. Anal.}, 8\penalty0 (3):\penalty0 411--436, 1998.
\newblock ISSN 1016-443X.
\newblock \doi{10.1007/s000390050062}.

\bibitem[Benjamini et~al.(2006)Benjamini, Kozma, and Wormald]{BKW}
Itai Benjamini, Gady Kozma, and Nicholas~C. Wormald.
\newblock {The mixing time of the giant component of a random graph}.
\newblock arXiv:0610459v1 [math.PR], 2006.

\bibitem[Bobkov et~al.(2006)Bobkov, Houdr{\'e}, and Tetali]{bht2000}
S.~G. Bobkov, C.~Houdr{\'e}, and P.~Tetali.
\newblock The subgaussian constant and concentration inequalities.
\newblock \emph{Israel J. Math.}, 156:\penalty0 255--283, 2006.
\newblock ISSN 0021-2172.
\newblock \doi{10.1007/BF02773835}.

\bibitem[Bollob{\'a}s(2001)]{Bollobas1981}
B{\'e}la Bollob{\'a}s.
\newblock \emph{Random graphs}, volume~73 of \emph{Cambridge Studies in
  Advanced Mathematics}.
\newblock Cambridge University Press, Cambridge, second edition, 2001.
\newblock ISBN 0-521-80920-7; 0-521-79722-5.
\newblock \doi{10.1017/CBO9780511814068}.

\bibitem[Ding et~al.(2010)Ding, Kim, Lubetzky, and Peres]{DKLP10}
Jian Ding, Jeong~Han Kim, Eyal Lubetzky, and Yuval Peres.
\newblock Diameters in supercritical random graphs via first passage
  percolation.
\newblock \emph{Combin. Probab. Comput.}, 19\penalty0 (5-6):\penalty0 729--751,
  2010.
\newblock ISSN 0963-5483.
\newblock \doi{10.1017/S0963548310000301}.

\bibitem[Dubhashi and Ranjan(1998)]{dr98}
Devdatt Dubhashi and Desh Ranjan.
\newblock Balls and bins: a study in negative dependence.
\newblock \emph{Random Structures Algorithms}, 13\penalty0 (2):\penalty0
  99--124, 1998.
\newblock ISSN 1042-9832.
\newblock \doi{10.1002/(SICI)1098-2418(199809)13:2<99::AID-RSA1>3.0.CO;2-M}.

\bibitem[Durrett(2010)]{Durrett-2007}
Rick Durrett.
\newblock \emph{Random graph dynamics}.
\newblock Cambridge Series in Statistical and Probabilistic Mathematics.
  Cambridge University Press, Cambridge, 2010.
\newblock ISBN 978-0-521-15016-3.

\bibitem[Janson and Luczak(2007)]{SJ184}
Svante Janson and Malwina~J. Luczak.
\newblock A simple solution to the {$k$}-core problem.
\newblock \emph{Random Structures Algorithms}, 30\penalty0 (1-2):\penalty0
  50--62, 2007.
\newblock ISSN 1042-9832.
\newblock \doi{10.1002/rsa.20147}.

\bibitem[Janson et~al.(1993)Janson, Knuth, {\L}uczak, and
  Pittel]{janson93birth}
Svante Janson, Donald~E. Knuth, Tomasz {\L}uczak, and Boris Pittel.
\newblock The birth of the giant component.
\newblock \emph{Random Structures Algorithms}, 4\penalty0 (3):\penalty0
  231--358, 1993.
\newblock ISSN 1042-9832.
\newblock \doi{10.1002/rsa.3240040303}.

\bibitem[Janson et~al.(2000)Janson, {\L}uczak, and Rucinski]{JLR}
Svante Janson, Tomasz {\L}uczak, and Andrzej Rucinski.
\newblock \emph{Random graphs}.
\newblock Wiley-Interscience Series in Discrete Mathematics and Optimization.
  Wiley-Interscience, New York, 2000.
\newblock ISBN 0-471-17541-2.
\newblock \doi{10.1002/9781118032718}.

\bibitem[Joag-Dev and Proschan(1983)]{jdp83}
Kumar Joag-Dev and Frank Proschan.
\newblock Negative association of random variables, with applications.
\newblock \emph{Ann. Statist.}, 11\penalty0 (1):\penalty0 286--295, 1983.
\newblock ISSN 0090-5364.
\newblock \doi{10.1214/aos/1176346079}.

\bibitem[Kolchin(1986)]{kolchin86mappings}
Valentin~F. Kolchin.
\newblock \emph{Random mappings}.
\newblock Translation Series in Mathematics and Engineering. Optimization
  Software Inc. Publications Division, New York, 1986.
\newblock ISBN 0-911575-16-2.
\newblock Translated from the Russian.

\bibitem[{\L}uczak(1990)]{luczak90component}
Tomasz {\L}uczak.
\newblock Component behavior near the critical point of the random graph
  process.
\newblock \emph{Random Structures Algorithms}, 1\penalty0 (3):\penalty0
  287--310, 1990.
\newblock ISSN 1042-9832.
\newblock \doi{10.1002/rsa.3240010305}.

\bibitem[{\L}uczak(1991{\natexlab{a}})]{Luczak}
Tomasz {\L}uczak.
\newblock Size and connectivity of the {$k$}-core of a random graph.
\newblock \emph{Discrete Math.}, 91\penalty0 (1):\penalty0 61--68,
  1991{\natexlab{a}}.
\newblock ISSN 0012-365X.
\newblock \doi{10.1016/0012-365X(91)90162-U}.

\bibitem[{\L}uczak(1991{\natexlab{b}})]{luczak91cycles}
Tomasz {\L}uczak.
\newblock Cycles in a random graph near the critical point.
\newblock \emph{Random Structures Algorithms}, 2\penalty0 (4):\penalty0
  421--439, 1991{\natexlab{b}}.
\newblock ISSN 1042-9832.
\newblock \doi{10.1002/rsa.3240020405}.

\bibitem[Newman and Watts(1999)]{Newman-Watts-1999}
M.~E.~J. Newman and D.~J. Watts.
\newblock Renormalization group analysis of the small-world network model.
\newblock \emph{Phys. Lett. A}, 263\penalty0 (4-6):\penalty0 341--346, 1999.
\newblock ISSN 0375-9601.
\newblock \doi{10.1016/S0375-9601(99)00757-4}.

\bibitem[Pavlov(1977)]{Pavlov77}
Yu.~L. Pavlov.
\newblock The asymptotic distribution of maximum tree size in a random forest.
\newblock \emph{Teor. Verojatnost. i Primenen.}, 22\penalty0 (3):\penalty0
  523--533, 1977 (Russian). 
\newblock English transl.: 
\newblock \emph{Th.  Probab. Appl.}, 22\penalty0  (3):\penalty0 509--520, 1977.
\newblock ISSN 0040-361x.

\bibitem[Pavlov(2000)]{Pavlov}
Yu.~L. Pavlov.
\newblock \emph{Random Forests}.
\newblock Karelian Centre Russian Acad. Sci., Petrozavodsk, 1996
(Russian).
\newblock English transl.: 
\newblock VSP, Zeist, The Netherlands, 2000.


\bibitem[Riordan and Wormald(2010)]{riordan08diam}
Oliver Riordan and Nicholas~C. Wormald.
\newblock The diameter of sparse random graphs.
\newblock \emph{Combin. Probab. Comput.}, 19\penalty0 (5-6):\penalty0 835--926,
  2010.
\newblock ISSN 0963-5483.
\newblock \doi{10.1017/S0963548310000325}.

\bibitem[Walkup(1980)]{walkup80}
David~W. Walkup.
\newblock Matchings in random regular bipartite digraphs.
\newblock \emph{Discrete Math.}, 31\penalty0 (1):\penalty0 59--64, 1980.
\newblock ISSN 0012-365X.
\newblock \doi{10.1016/0012-365X(80)90172-7}.

\bibitem[Watts and Strogatz(1998)]{Watts-Strogatz-1998}
D.J. Watts and S.H. Strogatz.
\newblock {Collective dynamics of `small-world' networks}.
\newblock \emph{Nature}, 393:\penalty0 440--442, 1998.

\end{thebibliography}

\end{document}